\documentclass[12pt]{article}
\usepackage{amsmath,amsfonts,amssymb,amsthm}
\usepackage{bbm}
\usepackage{graphicx}
\usepackage{tabularx}
\usepackage{color}
\definecolor{darkgreen}{rgb}{0,0.55,0}
\newcommand{\blue}[1]{{\textcolor{blue}{#1}}}

\newcommand{\grad}{\nabla}

\newcommand{\laplace}{\Delta}

\renewcommand{\div}{\grad\cdot}
\newcommand{\curl}{\grad\times}
\newcommand{\M}{\mathcal{M}}

\newcommand{\R}{\mathbbm{R}}

\newcommand{\Z}{\mathbbm{Z}}
\newcommand{\g}{\mbox{\sl g}}
\newcommand{\s}{\mbox{\sl s}}

\newcommand{\T}{\mathcal{T}}

\DeclareMathOperator{\dist}{dist}

\DeclareMathOperator{\gradient}{grad}

\DeclareMathOperator{\diff}{d}     

\newcommand{\K}{\ensuremath{\mathcal{K}}}

\newcommand{\Ha}{\ensuremath{\mathcal{H}}}

\newcommand{\mres}{%
  \,\raisebox{-.127ex}{\reflectbox{\rotatebox[origin=br]{-90}{$\lnot$}}}\,%
}

\newcommand{\G}{\ensuremath{\mathcal{G}}}
\DeclareMathOperator{\exc}{Exc}
\newcommand{\pe}{p^{\eps}}

\newcommand{\ue}{u^{\eps}}
\newcommand{\ve}{v^{\eps}}

\newcommand{\re}{\rho^{\eps}}
\newcommand{\Li}{L^{\infty}}
\newcommand{\ep}{\varepsilon}
\newcommand{\calH}{\Ha}
\newcommand{\eps}{\varepsilon}

\newcommand{\I}{\mathbbm{1}}

\newtheorem{prop}{Proposition}
\newtheorem{theorem}{Theorem}
\newtheorem{lemma}{Lemma}
\newtheorem{remark}{Remark}
\newtheorem{definition}{Definition}

\newcommand{\tacka}{\, \cdot\,}

\begin{document}

\title{On the vortex filament conjecture for Euler flows}

\author{
Robert L. Jerrard\thanks{
Department of Mathematics, University of Toronto, Toronto, Ontario M5S 2E4, Canada
} \and Christian Seis\thanks{
Institut f\"ur Angewandte Mathematik, Universit\"at Bonn,  53115 Bonn, Germany
}}
\maketitle

\begin{abstract}
In this paper, we study the evolution of a vortex filament in an incompressible ideal fluid. Under the assumption that the vorticity is concentrated along a smooth curve in $\R^3$, we prove that the curve evolves to leading order by binormal curvature flow. Our approach combines new estimates on the distance of the corresponding Hamiltonian-Poisson structures with stability estimates recently developed in \cite{JerrardSmets15}.
\end{abstract}

\section{Introduction}

In this paper, we study the evolution of an incompressible ideal fluid described by the Euler equations.  We are interested in data such that the initial vorticity is concentrated in a tube of radius $\eps\ll1$ around a smooth curve in $\R^3$.
%
One might then ask,
\begin{itemize}
\item  at later times, does the vorticity continue to
concentrate around some curve, and
\item  if so, how does the curve evolve?
\end{itemize}
The second question is not so hard {\em  if} one
already has a sufficiently good answer to the first. Indeed, 
the literature on this question, dating back to work of da Rios in 1906  \cite{DaRios06}, 
shows that if one somehow knows that at some time the vorticity concentrates smoothly
and  symmetrically in a
small tube around a smooth  curve, then one
can compute the instantaneous velocity of the curve to leading order. These computations suggest that the
curve should evolve, after a possible rescaling in time, by an equation described in Section \ref{ss.bn} below, known by various names, including the binormal curvature flow, the vortex filament equation, and the local induction approximation. 

Our results have the same character: they provide information
about curve evolution, conditional upon knowing that vorticity remains concentrated
around some curve. Their new feature is that we show that
for these purposes, a quite weak description of the vorticity concentration suffices. 
Roughly speaking, we show that for suitable initial data,  as long as the vorticity remains 
concentrated  on the {\em same scale}\  $\ep \ll 1$ around some curve of the correct arclength,
where concentration is measured by a geometric variant of
a particular negative Sobolev norm, 
then in fact the curve evolves by the binormal curvature flow.
Our results also improve on earlier work in that they require very little {\em a priori}
smoothness of the curve around which the vorticity concentrates, and
they apply to very rough solutions of the Euler equations.

From another perspective, the relationship between the Euler 
equations and the binormal curvature flow can be formally understood
by a correspondence between Hamiltonian-Poisson structures
giving rise to the two flows. Our results may be seen as
giving quantitative estimates of a sort of distance between
these Hamiltonian-Poisson structures.

The belief that one can find solutions of the Euler equations for which
the vorticity remains close for a significant
period of time to a filament evolving by binormal curvature flow
may be called the ``vortex filament conjecture" for the Euler equations.
We believe that our results provide more credible evidence
in favor of  the conjecture than any earlier arguments that we are aware of.


We conclude the introduction with a brief overview on preliminary and related works.
\begin{itemize}
\item {\em Formal asymptotics.} As mentioned earlier, the first derivation of the binormal curvature flow dates back to the work of da Rios in 1906 \cite{DaRios06}. 
In his doctoral thesis, the Italian mathematician formally computed the motion law of vortex filaments with the help of potential theory. At this time, da Rios' work was mostly ignored, except by his supervisor Levi-Civita, who promoted the results in a survey article \cite{Levi-Civita32a,Levi-Civita32b} many years later. In subsequent years, the local induction approximation was rediscovered several times, see \cite{Ricca91} and references therein, and it is by now a classical
topic in fluid dynamics.
Discussions that include alternative models arising from more refined formal asymptotics can
be found for example in the texts of Saffman \cite[Chapter 11]{Saffman92} or of Majda and Bertozzi \cite[chapter 7]{MajdaBertozzi02}.

\item {\em Rigid motion.} An explicit example for the motion of a vortex filament in an Euler fluid is the rigid motion of a perfect vortex ring, see e.g.\ \cite{FraenkelBerger74,AmbrosettiStruwe89}. Here, the evolution reduces to a translation with constant speed in direction  normal to the plane in which the ring is embedded. The possibility of non-trivial steady vortex configurations featuring knots and links was conjectured by Kelvin \cite{Kelvin1875}. Only very recently, such (infinite energy) solutions were found by Enciso and Peralta-Salas \cite{EncisoPeralta15,EncisoPeralta15b}. Explicit knotted solutions to the binormal curvature flow were studied in \cite{Kleener90}.
\item {\em Dimension reduction.}
To the best of our knowledge, the only rigorous result in favor of the vortex filament conjecture for Euler flows is restricted to flows with an axial symmetry. In \cite{BenedettoCagliotiMarchioro00}, the authors manage to show  that the (axially symmetric) vorticity remains sharply concentrated in a small tube which rigidly moves at a constant speed in the direction of the symmetry axis.
The analogous problem for two-dimensional fluids is much easier and is by now well-understood. In fact, if the vorticity is initially sharply concentrated around a number of points in $\R^2$ (or a subdomain), the 2D Euler dynamics are well described by the so-called point vortex model. For details, we refer the interested reader to Chapter 4 of the monograph \cite{MarchioroPulvirenti83} and the references therein.
\end{itemize}

The binormal curvature flow is also conjectured to arise as a description of dynamics of vortex filaments in certain quantum fluids,
as described by the Gross-Pitaevskii equation. This problem too is very largely open, although some conditional results, similar in spirit to the ones we prove here, are established in \cite{Jerrard02}. 

A higher-dimensional analog of the binormal curvature flow has been shown formally to describe the motion of
codimension 2 vortex submanifolds in ideal fluids in dimensions $n \ge 4$, first in the context of quantum fluids
\cite{Jerrard02}, and more recently for the Euler equations  \cite{Khesin12,Shashikanth12}.

This article is organized as follows. In Section \ref{ss.bn}, we introduce the Euler equations,  the binormal curvature flow, and the notation that is used in this paper. We subsequently present our rigorous results and discuss the method of this paper. In Section \ref{S.Hamiltonian}, we present a heuristic derivation of the binormal curvature flow from the Euler equations by formally passing to the limit in the corresponding Hamiltonian-Poisson structures.  The remaining Section \ref{S:Proofs} contain the proofs. 

\section{Mathematical setting and results}\label{ss.bn}


\subsection{Notation}

For notational convenience, throughout the manuscript we use the same notation for length, area, and volume. More precisely, $|\cdot|$ can stand for the one- or two-dimensional Hausdorff measure $\Ha^1$ or $\Ha^2$, respectively, or the Lebesgue measure $\mathcal{L}^3$ on $\R^3$. It should be clear from the context, which measure is actually used.

We will always write $\Gamma$ to denote a closed, oriented Lipschitz curve (normally smoother) in $\R^3$ of length $L$, and $\gamma:\R/L\Z\to \R^3$ 
an arclength parametrization of $\Gamma$. 
Thus $|\gamma'(s)|=1$ for all $s\in \R/L\Z$, 
\[
\Gamma = \left\{\gamma(s):\: s\in \R/L\Z\right\}\qquad\mbox{and }\ \gamma'(s) = \tau_\Gamma :=  \mbox{unit tangent to $\Gamma$ at $\gamma(s)$.}
\]
Depending on the context we may freely change between the notations $\partial_s\gamma $ and $\gamma'$ for the derivative of $\gamma$ with respect to the arc-length parameter.

For $s,t\in \R/L\Z$, we always understand $|s-t|$  to mean distance {\em modulo} $L\Z$
between $s$ and $t$; that is,
$|s-t| = \min_{k\in \Z}|s-t-Lk|$.

We remark that throughout most of the paper we will normalize by setting $L=1$.

We will write 
\[
A \lesssim B
\]
to mean that there exists some constant $C$, 
independent of 
$\ep$ (as long as $0<\ep<\frac  L 2$), such that
$A \le C B$.   Similarly, $A = \mathcal{O}(B)$  if $|A|\lesssim B$. 
Except where explicitly noted otherwise, the implicit
constants are {\em absolute} in the sense that they are
independent of all parameters.
The implicit constants appearing in all such estimates
may change from line to line.

We will write $\mu_\Gamma$ to denote the vector-valued measure corresponding to integration over $\Gamma$,
defined by
\[
\int \phi \cdot d\mu_\Gamma 
=  \int_{\Gamma} \phi\cdot \tau_{\Gamma}\, d\Ha^{1}
=
\int_{\R/L\Z} \phi(\gamma(s))\cdot \gamma'(s) \, ds
\qquad\mbox{ for }\phi \in C_c(\R^3;\R^3).
\]
If $\mu$ is an absolutely continuous measure with density $\omega$, i.e., $d\mu = \omega dx$, we will occasionally identify $\omega $ with $\mu $.

Derivatives of measures are defined in the sense of distributions.

Given $\gamma$ and $\Gamma$ as above, 
for $s\in \R/L\Z$, we define the {\em security radius} of $\Gamma$ at $\gamma(s)$ by
\begin{equation}
r_\gamma(s) = r(s) := \sup\left\{ r>0 \ : 
\begin{array}{c}\ |\gamma(s+h) -\gamma(s)|\ge  \frac r2 \ \mbox{ for all }|h|\ge r, \ \mbox{ and } \\
|\gamma'(s+h)-\gamma'(s) | \le \frac {|h|}r   \ \mbox{ for all }|h|\le r. \end{array}\right\}.
\label{r.def}\end{equation}
We set $r(s)=0$ if $\gamma$ is not differentiable at $s$.
We will see in Lemma \ref{L9cis} below that within the tube of variable radius $r_{\gamma} /4$ around $\Gamma$, there is a well-defined orthogonal projection onto $\Gamma$.

We will also write (omitting subscripts when no confusion can result)
\begin{equation}\label{kappastar}
\kappa_\gamma^*(s) = \kappa^*(s) := 1/r(s).
\end{equation}
It is easy to check that 
\begin{equation}
\kappa^*(s) \ge \kappa(s) = |\gamma''(s)| \quad\mbox{ for all $s$,}
\label{ksgpp}\end{equation}
provided that the letter is defined. Note that all arclength parametrizations $\gamma$ are translates of one another, so 
that quantities such as norms of $\kappa^*$ depend only on the geometry of $\Gamma$,
not on the parametrization.  In particular, we will be interested in Lipschitz curves $\Gamma$ for which
\begin{equation}
\| \kappa^*_\Gamma \|_{L^{1,\infty}}  = \| \kappa^* \|_{L^{1,\infty}} := \sup_{\sigma>0}
\sigma \Big| \big\{ s\in \R/L\Z :  |\kappa^*_\gamma(s)|\ge \sigma
\big\}
\Big| < \infty.
\label{kstarfinite}\end{equation}
This is a weak regularity condition that allows corners (but not cusps) and a finite number of self-intersections.

We will also need the following.
Let $\omega$ be a vector-valued Radon measure on $\R^3$. We define the {\em homogeneous flat norm} of $\omega$ as
\begin{equation}
\|\omega\|_F := \sup\left\{ \int \xi\cdot d\omega:\: \xi \in C^1_c(\R^3;\R^3)\mbox{ with }\|\curl\xi\|_{L^{\infty}}\le1\right\}.
\label{def1}\end{equation}
If $\omega:\R^3\to\R^3$ is a locally integrable vector field, then we write $\|\omega\|_F= \|\mathcal{L}^3\mres \omega\|_F$, where $\mathcal{L}^3$ is the Lebesgue measure on $\R^3$.

%

\subsection{Euler equations}

We consider the time-rescaled Euler equations
\begin{eqnarray}
k_{\eps}^{-1}\partial_t \ue + \ue\cdot \grad \ue  +\grad \pe &=&0,\label{1}\\
\div \ue&=& 0,\label{2}
\end{eqnarray}
where, as usual,
$\ue:[0,\infty)\times \R^3\to \R^3$ denotes the fluid velocity 
and $\pe:[0,\infty)\times \R^3\to \R$ is the pressure. Moreover, $k_{\eps}$ is a scaling factor that has to be specified later. The initial configuration is a divergence-free vector field $\ue_0$ on $\R^3$, that is
\[
\ue(0, \tacka)=\ue_0.
\]
To be more specific, we are interested in {\em weak solutions} of the Euler equation.

\begin{definition}\label{D:weak_Euler} We call $\ue\in L^\infty(\R; L^2( \R^3; \R^3))$ a {\em weak solution} of \eqref{1}, \eqref{2} if
$\nabla\cdot \ue = 0$ in the sense of distributions, and 
\begin{equation}
\int_0^{\infty}\int k_{\eps}^{-1} \partial_t \phi \cdot \ue  + \nabla \phi : (\ue\otimes \ue) \ dx\,dt + \int \phi(0,\tacka)\cdot \ue_0\, dx=0
\label{wk.sol}\end{equation}
for every test function $\phi\in C^\infty_c(\R\times \R^3; \R^3)$ such that $\nabla\cdot \phi = 0$. 
A weak solution is said to be {\em conservative} if 
\[
t\mapsto \int |\ue(t,x)|^2 dx \qquad \mbox{ is constant}.
\]
\end{definition}
Here and in the following, we understand undetermined integrals as integrals over the whole space.

It is well-known that in \eqref{wk.sol} the pressure can be reintroduced as a Lagrange multiplier for the divergence free condition on the test functions, and is then uniquely determined up to a function that depends on time $t$ only. The weak solutions $\ue$, however,  are not unique, see for example \cite{Scheffer93,Shnirelman97,Shnirelman00, DeLellisSzekelyhidi09, BdLIsSz}.
In particular, it is shown in \cite{DeLellisSzekelyhidi10} that energy conservation fails as criterion for uniqueness. On the positive side existence of weak solutions for any initial datum in $ L^2$ was established in \cite{Wiedemann11}. These solutions are non-conservative (in fact, the energy is even discontinuous). In \cite{SzekelyhidiWiedemann12}, the authors construct a dense subset of $L^2$ for which conservative solutions exist. 

From the definition of weak solutions in \eqref{wk.sol}, we immediately infer the following identity.

\begin{lemma}\label{L1}
Let $u\in L^2_{loc}(\R\times \R^3;\R^3)$ be a weak solution to the Euler equation, and
let  $\omega=\curl u$ the vorticity (measure). Then for every $\phi \in C_c^{\infty}(\R^3;\R^3)$
\begin{equation}\label{8}
\frac{d}{dt}\int \phi\cdot d \omega = k_{\eps}  \int \grad(\curl\phi):u\otimes u\, dx
\qquad\mbox{ distributionally in }(0,\infty).
\end{equation}
\end{lemma}

\begin{proof}
Fix $\phi \in C^\infty_c(\R^3;\R^3)$ and $f\in C^\infty_c(0,\infty)$. Since $u$ is a weak solution,
\[
\int  f'(t) \left(\int  (\nabla\times \phi) \cdot u \ dx \right) dt    =   -k_{\eps} \int f(t) 
\left(\int \nabla(\nabla\times \phi) : (u\otimes u)\, dx\right) dt ,
\]
and this is exactly \eqref{8}.
\end{proof}

\subsection{Binormal curvature flow}

A family of  smooth curves $\{\Gamma(t)\}_{t\in[0,T]} \in \R^3$ is said to evolve by binormal curvature flow (BCF) if
\begin{equation}\label{6a}
\partial_t \gamma = \gamma' \times \gamma'',
\end{equation}
where for each $t\in[0,T]$, $\gamma(t,\tacka):  \R/L\Z\to \R^3$ is an arc-length parametrization of $\Gamma(t)$, i.e., $\Gamma(t) = \left\{ \gamma(t,s):\: s\in \R/L\Z\right\}$ and $|\gamma'(t,s)|^2 = 1$ for all $s\in \R/L\Z$. Here, $L$ is the length of the curve, and we assume the curve to  be closed, so that $\R/L\Z$ is the interval of periodicity of $\gamma(t,\tacka)$. In case where $\Gamma(t)$ is a Frenet curve, then we can equivalently write \eqref{6a} as
\[
\partial_t \gamma = \kappa b,
\]
where $\kappa$ is the curvature and $b$ is the binormal vector along the curve. That arc-length parametrizations are indeed compatible with binormal curvature flows can be seen by computing
\[
\frac{\partial}{\partial t} |\gamma'|^2 = 2\gamma'\cdot\partial_t\gamma' \stackrel{\eqref{6a}}{=}0.
\]
Short-time existence of smooth solutions follows from classical arguments.

%
%
%
%
%
%
%
%
%
%
%

There is a striking similarity between Lemma \ref{L1} for Euler solutions and the following formula for binormal curvature flows.

\begin{lemma}\label{L2} Let $\{\Gamma(t)\}_{t\in[0,T]}$ be a family of smooth curves evolving by binormal curvature flow. Then
\begin{equation}\label{7}
\frac{d}{dt} \int_{\Gamma} \phi\cdot \tau d\Ha^1 =  
\int_{\Gamma} \grad\left(\curl \phi\right): (I- \tau\otimes \tau)\, d\Ha^1,
\end{equation}
for all $\phi \in C_c^{\infty}(\R^3;\R^3)$. Here $\tau = \tau_{\Gamma}$ denotes the tangent along $\Gamma$.
\end{lemma}

In Section \ref{S.Hamiltonian} below, we will see that the identities \eqref{8} and \eqref{7} give rise to the Hamiltonian-Poisson structures of the Euler equation and the binormal curvature flow, respectively. In a certain sense, in Theorems \ref{T1} and \ref{T2} below, we will estimate the distance of these structures.

We briefly recall the proof of Lemma \ref{L2} from \cite{JerrardSmets15}; see also \cite{Jerrard02} for a more general result.

\begin{proof} Let $\gamma$ be an arc-length parametrization of $\Gamma$ satisfying \eqref{6a}. 
We write $\phi$ and  $\gamma$ in components and compute (summing implicitly)
\[
\frac d{dt}
\int_{\R/L\Z} \phi_i(\gamma)\, \partial_s\gamma_i \, ds 
=  \int_{\R/L\Z} \partial_j\phi_i(\gamma)\partial_t \gamma_j \partial_s\gamma_i \,  +
\phi_i(\gamma)\partial_s\partial_t\gamma_i \,ds .
\]
Integrating by parts in the second term on the right-hand side and rearranging,
one finds that
\[
\frac d{dt}
\int_{\R/L\Z} \phi_i(\gamma)\,\partial_s\gamma_i \, ds 
= \int_{\R/L\Z} (\nabla \times \phi)(\gamma) \cdot (\partial_t\gamma \times \partial_s\gamma)\, ds.
\]
Since $\partial_s \gamma\cdot \partial_{ss} \gamma=0$, the equation  \eqref{6a} implies that
$
\partial_t\gamma \times \partial_s\gamma = - \partial_s\gamma \times \partial_t\gamma
= \partial_{ss}\gamma$. One arrives at \eqref{7} by substituting this into the right-hand side, integrating by parts again, and
using the fact that $\nabla(\nabla\times \phi)$
is trace-free, so that $\nabla(\nabla\times \phi): I = 0$.
\end{proof}
 
In \cite{JerrardSmets15}, Smets and the first author develop a notion of 
weak solutions of the binormal curvature flow, in the spirit of
geometric measure theory, based on the identity \eqref{7},
and allowing for phenomena, such as changes of topology, 
seen in vortex filaments in real fluids.
In the  present work, \eqref{6a} is a suitable notion as we deal with smooth flows only. Still, the results from \cite{JerrardSmets15} enter into our analysis through stability estimates in the spirit of Theorem 3 in \cite{JerrardSmets15}. 
 
%
%
%
%
%
%
%
%
%
%
%

%
%
%
%
%
%
%
%

%
%
%
%
%
%
%
%
%
%


\subsection{Main results}\label{S:Main}

In this section we state our two main results. First, we give some conditions under which it can be shown that  Euler vortex filaments evolve to leading order by binormal curvature flow. The first condition is that the vorticity concentrates on an $\eps$-scale around a smooth curve of fixed length. Concentration is measured in terms of the flat norm, introduced in \eqref{def1} above. More precisely, we will be interested in velocity vector fields
such that 
\begin{equation}
\label{7b}
\| \nabla\times \ue - \mu_{\Lambda}\|_F \le \ep L 
\end{equation}
for some curve $\Lambda$ of length $L$, satisfying the weak regularity condition
\eqref{kstarfinite}. 
We will show that \eqref{7b} implies a lower bound on the kinetic energy:
\begin{equation}
\int\frac 12  |\ue|^2 \ge \frac {L \log(L/\ep)}{4\pi} \  - \mathcal{O}(1).
\label{elbd}\end{equation}
This is contained in Theorem \ref{T1} below,
though with a very indirect proof.
We next fix the time rescaling factor in the Euler equations as
\[
k_{\eps} = \frac{4\pi}{|\log(\eps/L)|}.
\]
It is a classical fact that this is the ``right" scaling to obtain binormal curvature flow in
the limit $\ep\to 0$.
As we discuss in Section \ref{S.Hamiltonian} below,
this fact follows formally from the energy scaling \eqref{elbd},
and it is confirmed by our main results.

We also introduce the excess of the kinetic energy relative to the curve $\Lambda$:
 \begin{equation}\label{exc.def}
\exc_\ep(\ue,\Lambda) := \frac{k_{\eps}}{L}\int \frac 12 |\ue|^2\, dx -1,  \qquad\quad L = |\Lambda| . 
\end{equation}
The excess is a dimensionless quantity and is preserved by the evolution if $\ue$ is a conservative solution and $\Lambda$ has constant length (e.g., as a solution to the binormal curvature flow). 
It measures the extent to which the lower bound
in \eqref{elbd} is saturated. We will be interested in velocity fields for which
\begin{equation}
\label{7c}
\exc_{\eps}(\ue,\Lambda)\le C k_{\eps}.
\end{equation}
Together, conditions \eqref{7b}, \eqref{7c} imply that 
the kinetic energy is essentially induced by vorticity.

Our first main result 
 estimates the difference between the right-hand
sides of identities \eqref{8} and \eqref{7} for
Euler and the binormal curvature flow.
This can be understood as an estimate of  the extent to 
which the (distributional) instantaneous velocity
of a vortex filament in a solution of the Euler equations
deviates from the binormal curvature.

\begin{theorem}\label{T1}
Let $\Gamma\subset \R^3$ be an oriented closed Lipschitz curve of length $L$ and let $\gamma:\R/L\Z\to \R^3$
be an arclength parametrization of $\Gamma$ 
such that
\[
\| \kappa^*\|_{L^{1,\infty}} < \infty.
\]
For $\eps\in(0,\frac{L}2)$, let $\ue\in L^2(\R^3;\R^3)$ be divergence-free vector fields such that $\ue$ and $\Gamma$ satisfy condition \eqref{7b}. Then
\[
0\le \exc_\eps(u_{\eps},\Gamma) + \mathcal{O}\left(\|\kappa^*\|_{L^{1,\infty}}^2k_{\eps}\right),
\]
and there exists an absolute constant $C<\infty$ such that
\begin{eqnarray}
\lefteqn{\frac1L\left| k_{\eps} \int \phi : \ue\otimes\ue\, dx - \int_{\Gamma}\phi:\left(I-\tau\otimes\tau\right)\, d\Ha^1 \right|}\nonumber\\
&\le& 
C \|\phi\|_{L^{\infty}}\exc_{\eps}(\ue,\Gamma) + \mathcal{O}\left( k_{\eps} \|\kappa^*\|_{L^{1,\infty}}^2\|\phi\|_{W^{1,\infty}_L}\right)
\label{main.estimate}
\end{eqnarray}
for all $\phi\in W^{1,{\infty}}(\R^3;M^{3\times 3})$, where $M^{3\times 3}$ is the space of $3\times 3$ matrices,
$\| \phi \|_{W^{1,\infty}_L} := \|\phi\|_{L^{\infty}} + L \|\nabla \phi\|_{L^{\infty}}$.
\end{theorem}

In view of Lemmas \ref{L1} and \ref{L2}, the conclusion \eqref{main.estimate}
implies that if $\ue(x,t)$ is a solution of the Euler equation, $\{ \Gamma(t)\}$ is a
binormal curvature flow of length $L$, 
and if $\| \nabla \times \ue(\cdot,t_0) - \mu_{\Gamma(t_0)}\| \le \ep L$ at some time $t_0$,
then
\[
\left| \frac d{dt} \int \nabla\times \ue \cdot \phi - \frac d{dt} \int _{\Gamma(t)} \phi\cdot \tau
\right| \le 
C \|\phi\|_{L^{\infty}}\exc_{\eps}(\ue,\Gamma) + \mathcal{O}\left( k_{\eps} \|\kappa^*\|_{L^{1,\infty}}^2\|\phi\|_{W^{1,\infty}_L}\right)
\]
at time $t_0$. This shows that (at time $t_0$) the vorticity in $\ue$ is
close in a distributional sense to a binormal curvature flow if the excess is small.

The same estimate \eqref{main.estimate} may also be understood as describing the distance of the Hamiltonian-Poisson structures associated with the Euler equation, cf.\ \eqref{8}, and the binormal curvature flow, cf.\ \eqref{7}. We discuss this in more detail in Section \ref{S.Hamiltonian}, where we also review the Hamiltonian-Poisson structures of both evolutions.

Like all prior work\footnote{Here and in what follows, we are omitting  papers 
\cite{ FraenkelBerger74,AmbrosettiStruwe89,BenedettoCagliotiMarchioro00} that assume rotational and often additional symmetries.} on this subject, Theorem \ref{T1} addresses only the
question of estimating the instantaneous velocity of a vortex filament
known to concentrate at a fixed time $t_0$ around a curve -- it does not say anything about when
such a concentration condition is preserved by the dynamics. But in
a number of other ways, it  improves on known results:

\begin{itemize}
\item
As far as we know, all previous studies of
the dynamics of vortex filaments consist of asymptotic computations
that describe only  {\em highly idealized} vortex filaments, such as 
the ``prototype velocity field" associated to an $\ep$-regularization
of a $C^2$ filament, introduced in Section \ref{S.method} below.

By contrast, Theorem \ref{T1} 
applies to a much larger and more physically reasonable class of velocity fields
 --- those with vorticity concentrated
in the weak sense \eqref{7b} about some curve $\Gamma$, 
and with small excess.

\item Earlier results that we are aware of do not obtain any very useful control over error terms, whereas Theorem \ref{T1} quantifies errors well enough to conclude in Theorem \ref{T2} below  that, as long as
the vorticity remains concentrated around {\em some} curve,
one can control the closeness
of the vorticity to a binormal curvature 
flow over a macroscopic time interval.

One reason this is possible is that the distributional estimates that we obtain, relating vortex
filament velocity and the binormal curvature flow,
seem to be more useful than the pointwise estimates found in earlier work.

\item Theorem \ref{T1} shows that, at least at a fixed time, the binormal curvature flow 
approximates the velocity of vortex filaments (in a distributional sense), even when the vorticity is concentrated around a curve of low regularity,
measured by the geometric quantity $\| \kappa^*\|_{L^{1,\infty}}$. 

For example, a recent paper \cite{delaHozVega} of de la Hoz and Vega studies the binormal
curvature flow for initial data given by a regular planar polygon.
Owing to the weak regularity conditions imposed on the
curve $\Gamma$, Theorem \ref{T1} implies that if one considers the Euler equation
for initial data whose vorticity is concentrated in the sense of \eqref{7b}
around a polygon, and with small excess, then the distributional
initial velocity of the vorticity is close to the distributional
binormal curvature of the polygon -- 
a sum of delta functions at its vertices.

\end{itemize}

In short, to the best of our knowledge, this is the first time that dynamics of  vortex filaments in Euler flows have been approached in a quantitative way. 

Our second main result  shows, as discussed above, that a
distributional estimate such as \eqref{main.estimate}
is sufficient to ensure that a vortex filament remains close to
a binormal curvature flow over a macroscopic time interval ({\em i.e.}
bounded below, independent of $\ep$, as $\ep\to 0$).

\begin{theorem}
Let
$\{\Gamma(t)\}_{t\in[0,T]}$ be a family of smooth curves evolving by binormal curvature flow \eqref{6a} with $|\Gamma(t)|=L$ and with arc-length parametrizations $\gamma  : [0,T]\times \R/L\Z \to \R^3$  satisfying
\begin{equation}\label{Rgamma}
\sup_{0\le t\le T} \|\partial_s^4\gamma\|_{L^\infty} <\infty, \qquad\qquad \inf_{0\le t\le T,s\in \R/L\Z}  r_{\gamma(t,\cdot)} (s)  >0.
\end{equation}
For $\eps\in(0,\frac{L}2)$, let $\ue$ be a conservative weak solution to the Euler equation \eqref{1}, \eqref{2}, for initial data satisfying 
\begin{equation}
\| \nabla\times \ue(0) - \mu_{\Gamma(0)}\|_F \le \ep L,
\qquad
\exc_{\eps}(\ue(0),\Gamma(0))\le C k_{\eps}.
\label{initial.ue}\end{equation}

If, for every $t\in [0,T]$, there is a Lipschitz curve $\Lambda(t)$ with arc-length parametrization $\lambda(t,\tacka):  \R/L\Z \to \R^3$ and such that 
$\ue(t)$ and $\Lambda(t)$ satisfy \eqref{7b},
and if in addition  
\begin{equation}\label{kstar.uniform}
\mbox{ $|\Lambda(t)|=L$ for all $t$,\  } \qquad\qquad \sup_{t\in[0,T]}\| \kappa^*_{\Lambda(t)}(t,\cdot) \|_{L^{1,\infty}} <\infty \,,
\end{equation}
then there exists a function 
$\bar\sigma:[0,T]\to \R/L \Z$ and a constant $C'<\infty$, depending only
the bounds in \eqref{Rgamma}, \eqref{initial.ue},  \eqref{kstar.uniform},
and in particular independent of $\ep$ and $L$,
such that for any $t\in[0,T]$,
\begin{eqnarray}
\max_{s \in \R/L\Z}L^{-1} \left|\lambda(t,s) - \gamma(t,s +\bar \sigma(t)) \right| &\le& C' k_{\eps}^{1/2},\label{est1}\\
\left(L^{-1}\int_{\R/L\Z} |\partial_s\lambda(t,s) - \partial_s\gamma(t, s+\bar \sigma(t))|^2 \, ds\right)^{1/2} &\le& C' k_{\eps}^{1/2}\label{est2}.
\end{eqnarray}
\label{T2}\end{theorem}

Our statement contains two estimates. First, \eqref{est1} shows that in the regime of small $\eps$, the curves $\Lambda$ and $\Gamma$ are close one to the other uniformly in $t$ and $s$. The distance of both curves is of order $|\log\eps|^{-1/2} $ and thus controlled by the initial datum via \eqref{initial.ue}. The somewhat weaker $L^2$ bound on the difference of the tangents at $\Lambda$ and $\Gamma$ is displayed in \eqref{est2}. The latter ensures the closeness of both curves in a very geometric sense: The curves are locally very
nearly parallel. 
Because no assumptions are  imposed on the arc-length parametrizations $\lambda$, the spatial translations $\bar \sigma(t)$ are necessary in both \eqref{est1} and \eqref{est2}.


As remarked above, the existence of smooth solutions of the binormal curvature flow \eqref{6a} is standard, so hypotheses \eqref{Rgamma}, \eqref{initial.ue} are assumptions about the initial data. 
The content of the theorem is that to prove the vortex filament conjecture for
such initial data, it suffices to find a solution $\ue$
for which vorticity remains concentrated around {\em some} curve,
in the sense of \eqref{7b}, \eqref{kstar.uniform}. 
As far as we know, this is the first result to describe any conditions under which
vortex filaments can be related to the binormal curvature flow for a
macroscopic length of time. In addition,
Theorem \ref{T2}, like Theorem \ref{T1} above, is quantitative in a way that
one does not find in the previous literature.

The task of constructing solutions as considered in Theorem \ref{T2}, should one hope to do so, 
is in principle made easier by the 
the facts that $\ue$ need only belong to $L^2$, that vorticity concentration
is required only in the weak sense \eqref{7b}, and that only the
weak regularity condition
\eqref{kstar.uniform} is {\em a priori} required for
the curves $\Lambda(t)$.
For example, it is conceivable that one could construct
weak solutions  satisfying \eqref{7b}, \eqref{kstar.uniform} using 
techniques inspired by convex integration, as developed in references
such as
\cite{DeLellisSzekelyhidi09,DeLellisSzekelyhidi10, BdLIsSz, Wiedemann11, SzekelyhidiWiedemann12}. On the other hand, the theorem
shows that the regime described by \eqref{7b}, \eqref{kstar.uniform}
is quite ``rigid", and so may well be inaccessible to these techniques,
which exploit ``flexible" aspects of the Euler equations.

\begin{remark}\label{R:rescale}
For $\alpha>0$, the statements of the theorems are invariant with respect to the rescaling
\[
\Gamma \mapsto \tilde \Gamma =  \alpha \Gamma, \qquad
\ep \mapsto \tilde \alpha = \alpha\ep, \qquad
u(\cdot) \mapsto \tilde u(\cdot)  = \frac 1\alpha u\big(\frac \cdot \alpha \big) .
\]
Due to this scaling invariance, it  suffices to  prove the Theorems \ref{T2} and \ref{T1} for $|\Gamma|=1$.

In particular, we remark that if we write $\kappa^*$ and $\tilde \kappa^*$ for the functions defined in \eqref{kappastar},
associated to parametrizations of $\Gamma$ and $\tilde \Gamma$ respectively, then
it is straightforward to check that
$\| \kappa^* \|_{L^{1,\infty}} = \| \tilde \kappa^* \|_{L^{1,\infty}} $.
We believe that this invariance makes $\| \kappa^*\|_{L^{1,\infty}}$ 
a natural quantity to consider.
\end{remark}

\subsection{Method of this paper}\label{S.method}
 In this subsection, we explain the ideas of the proofs of our main results, Theorems \ref{T2} and \ref{T1} above. In view of Remark \ref{R:rescale}, it is enough to consider the case $L=1$ in the following.
 
The main estimate \eqref{main.estimate} of Theorem \ref{T1} is derived by estimating the (Euler) velocity $\ue$ against a prototype velocity field $\ve$, which itself satisfies the assertions of the theorem. To be more specific, given the curve $\Gamma$ from the hypothesis, we construct $\ve$ by
\[
\ve = \curl(-\laplace)^{-1}(\rho^{\eps}\ast \mu_{\Gamma}).
\]
Here, $\{\rho^{\eps}\}_{\eps\downarrow0}$ denotes a sequence of radially symmetric standard mollifiers in $\R^3$, supported in a ball of radius $\eps$ and such that $\rho^{\eps}(x)  = \eps^{-3}\rho^1(\eps^{-1}x)$. The convolution of $\rho^{\eps}$ with a measure $\mu$ is defined as
\[
\rho^{\eps}\ast \mu(x) =  \int_{\R^3}\rho^{\eps}(x-y)\, d\mu(y) .
\]
Moreover, the nonlocal differential operator $\curl (-\laplace)^{-1}$ associates to a vorticity field $\omega$ a vector potential $v$ via the Biot--Savart law
\begin{equation}
\label{57b}
v(x) = \curl(-\laplace)^{-1}\omega (x)  = \frac1{4\pi} \int \frac{x-y}{|x-y|^3}\times \omega(y)\, dy.
\end{equation}
That is, $v$ is the unique divergence free vector field satisfying $\curl v=\omega$. Notice that the Biot--Savart kernel is obtained as the curl of the Newtonian potential. We will sometimes write $(\curl)^{-1} = \curl(-\laplace)^{-1}$. The vector field $\ve$   can thus be written as
\[
v^{\eps}(x)=\frac1{4\pi} \int \int_0^1\rho^{\eps}(x-y) \frac{\gamma(s)-y}{|\gamma(s)-y|^3}\times\gamma'(s)\, ds\,dy
\]
where, as above, $\gamma:\R/\Z\to \R^3$ is an arclength parametrization 
of a smooth non-self-intersecting curve $\Gamma$.

The main properties of $\ve$ that we will use later on are summarized in the following proposition.

\begin{prop}\label{Prop.ve}
The following estimates hold:
\begin{gather}
\|\curl \ve - \mu_{\Gamma}\|_F\le \eps,
\label{ve.0}\\
\| \ve\|_{L^q}  \lesssim   \ep^{\frac 2q - 1} \| \kappa^*\|_{L^{1,\infty}} \quad \mbox{for all }q\in (2,\infty],
\label{ve.1}\\
\| \ve\|_{L^2}^2 = \frac {1}{2\pi} |\log\ep| + 
\mathcal{O}\left(\|\kappa^*\|_{L^{1,\infty}}^2 \right),
\label{ve.2}
\end{gather}
and
\begin{equation}
\frac{4\pi}{|\log \eps|}\int \phi: v^{\eps}\otimes v^{\eps}\, dx  =  \int_{\Gamma} \phi:(I-\tau\otimes\tau)\, d\Ha^1 
+ \mathcal{O}\left(\frac{\|\kappa^*\|_{L^{1,\infty}}^2 \|\phi\|_{W^{1,\infty}}}{|\log\eps|}\right),
\label{ve.4}\end{equation}
for all $\phi\in W^{1,\infty}(\R^3; M^{3\times 3})$, where $M^{3\times 3}$ denotes the space of $3\times 3$ matrices.
\end{prop}

Hence, in view of \eqref{ve.0}, \eqref{ve.2} and \eqref{ve.4}, the prototype velocity field $\ve$ satisfies, among others, the concentration condition \eqref{7b} and the two statements of Theorem \ref{T1}. Estimates \eqref{ve.1}--\eqref{ve.2} will be used to bound error terms. In order to verify Theorem \ref{T1}, we thus need to control how much $\ue$ deviates from $\ve$. This is the content of Section \ref{Proof.Theorem2} below. The concentration condition \eqref{ve.0} will be established in Section \ref{S:Flatnorm} below. The remaining statements of Proposition \ref{Prop.ve} are proved in Section \ref{S:prototype}.

Let us now discuss our approach to Theorem \ref{T2}. The hypothesis and the Hamiltonian structure of the Euler equations guarantee that the excess $\exc_{\eps}(\ue,\Lambda)$ is of the order of the error term, that is $\mathcal{O}(|\log\eps|^{-1})$. In view of the main estimate \eqref{main.estimate} of Theorem \ref{T1} and the weak formulation of Lemma \ref{L2} for the binormal curvature flow, the question of how close $\{\Lambda(t)\}_{t\in[0,T]}
$ is to the binormal curvature flow $\{\Gamma(t)\}_{t\in[0,T]}$ amounts thus to the study of stability of the latter. Stability properties of binormal curvature flows have been recently investigated by Smets and the first author. In \cite{JerrardSmets15}, it is shown that for a certain decaying $C^2$ continuation $X$ of $\tau_{\Gamma}$ it holds that
\begin{equation}
\label{rprop}
|\partial_t X\cdot \xi - \grad(\curl X):\xi\otimes \xi|\le K(1-X\cdot \xi),
\end{equation}
for some $K=K(\Gamma)$ and {\em all} unit vector fields $\xi$. This remarkable property is reminiscent of \eqref{7}. The {\em binormal curvature flow defect} of the flow $\{\Lambda(t)\}_{t\in[0,T]}$ is thus controlled by the distance of the tangent fields $\tau_{\Lambda}$ and $\tau_{\Gamma}$. We use this stability estimate to deduce Theorem \ref{T2} from Theorem \ref{T1}.

\section{The Hamiltonian--Poisson structure: Heuristics for the vortex filament conjecture}\label{S.Hamiltonian}

In this section, we review Marsden and Weinstein's interpretation of the Euler equations and the binormal curvature flow as Hamiltonian systems on Poisson manifolds \cite{MarsdenWeinstein83}, see also \cite{Shashikanth12, Khesin12} for more
recent and more general discussions. As opposed to the original paper, where this interpretation is worked out in the language of Lie algebras, our approach is based on vector calculus which allows us to identify the involved ingredients directly with various quantities that appear in the statements of our theorems. The parallels between the Hamiltonian-Poisson structures can be used to pass to the limit in a formal way. We perform this limit at the end of this section.



We recall that we need three ingredients to constitute a Hamiltonian-Poisson system:
\begin{itemize}
\item a differentiable manifold $\M$;
\item a Poisson bracket on $\M$, that is, a bilinear map that takes functions on $\M$ to functions on $\M$,  is skew-symmetric, i.e., $\{F,G\}+ \{G,F\}=0$, satisfies the Jacobi identity, i.e., $\{E, \{F,G\}\} + \{F, \{G,E\}\} + \{ G, \{E, F\}\}=0$, and obeys the Leibniz rule, i.e., $\{EF,G\} = \{E,G\} F + \{F,G\}E$;
\item and a function $H: \M\to \R$.
\end{itemize}
We call a dynamical system $p(t)$ in $\M$ that is given by the differential equation
\begin{equation}\label{100}
\frac{d}{dt} F(p) = \{F,H\}(p)
\end{equation}
for all $F:\M\to \R$ a {\em Hamiltonian-Poisson system} on $\M$. The function $H$ is referred to as the {\em Hamiltonian}. 
Observe that the Hamiltonian is preserved during the evolution, because
\[
\frac{d}{dt} H(p) \stackrel{\eqref{100}}{=}  \{ H,H\}(p)= 0
\]
by the skew-symmetry of the Poisson bracket.

In some situations, the Poisson bracket is induced by a skew-symmetric bilinear form on a Riemannian manifold $\s_p: \T_p\M\times \T_p\M\to \R$ via
\[
\{F,G\}(p)  = \s_p(\gradient F,\gradient G).
\] 
We also recall that the relation between gradients and differentials on Riemannian manifolds is established by the duality
\[
\g_p(\gradient H(p),\delta p)
= \diff H(p) (\delta p) 
\]
for all $\delta p\in \T_p\M$. If $\s_p$ is non-degenerate, $(\M,\s)$ is called a symplectic manifold.


We first explain the {\em formal} Hamiltonian-Poisson structure of the Euler equations. As we are interested in the situation where the vorticity is concentrated along curves, it is convenient to consider vorticity fields as  main objects.  
We accordingly take the manifold to be the phase space of vorticity fields
\[
\M= \left\{ \omega:\: \exists u\mbox{ s.t.\ } \div u=0\mbox{ and } \omega =  \nabla\times u \right\}. 
\]
%
%
%
For $\omega\in \M$ we will sometimes use the notation
$(\curl)^{-1}\omega = \nabla\times (-\Delta)^{-1}\omega\, (=u)$, as motivated by the Biot--Savart law \eqref{57b}.
Recall that this defines a canonical isomorphism between $\M$ and
the space of divergence-free velocity fields.
Since $\M$ is a linear space, we can identify the tangent space $T_\omega\M$ with
$\M$ itself. It turns out to be more convenient, however, to use the above-mentioned
isomorphism to represent the tangent space as the space of velocity fields
\[
\T_{\omega}\M \cong \left\{   v : \: \nabla\cdot v =0 \,\right\}. 
\]
This point of view is quite natural since vortex filaments are convected by the flow.
Thus, the tangent vector at $t=0$ to a smooth path $\omega(t)$ in $\M$ (with $\omega(0)=\omega$)
will be represented as
$v = (\nabla\times)^{-1}\dot\omega(0)$.
The Poisson bracket for the Euler equation is induced by a skew-symmetric bilinear (symplectic) form on a Riemannian metric. The metric tensor is then chosen to be 
\[
\g_{\omega}(v,w) = \int  v\cdot w \, dx,
\]
and the skew-symmetric bilinear form is
\[
s_{ \omega}(v,w) =   \int \left( v\times w \right)\cdot\omega\, dx.
\]
Finally, the Hamiltonian is the kinetic energy, which, when expressed in terms of the vorticity field is
 $H(\omega)= \frac12\int |(\curl)^{-1} \omega|^2\, dx$, which is just the $H^{-1}$ norm squared.

A short computation shows that the Hamiltonian--Poisson system associated with these ingredients constitutes the Euler equations. Indeed, the definitions imply that 
$\gradient H(\omega) = (\curl)^{-1}\omega =:u$, so
\[
\{F,H\}(\omega) = \int \left( \gradient F \times  u\right)\cdot \omega\, dx,
\]
and thus the elementary identity $A\cdot( B\times C) = -B\cdot (A\times C)$
and the fact that $\gradient F$ is divergence-free imply that 
\[
\{F,H\}(\omega)  
=- \int \gradient F \cdot \left(\omega\times u\right)\, dx =-
\int \gradient F \cdot \left(\omega\times u + \nabla q \right)\, dx 
\]
for a scalar $q$. In particular, since
\[
\omega\times u = u\cdot \nabla u - \nabla (\frac 12 |u|^2) 
\]
we can write 
\[
\omega\times u +\nabla q=  u\cdot\nabla u +\nabla p
\]
for $\nabla p$ chosen so that this whole expression is divergence-free
and  belongs to $T_\omega\M$.
Then we can rewrite the above as
\[
\{F,H\}(\omega)  
= -  \diff F(\omega) (u\cdot\nabla u + \nabla p).
\]
On the other hand, since we have identified $T_\omega\M$ with $(\curl)^{-1}\M$,
\[
\frac{d}{dt} F(\omega) = 
\diff F(\omega)\left((\curl)^{-1}\partial_t \omega \right)
=
\diff F(\omega)\left( \partial_t u \right).
\]
It follows that the evolution \eqref{100} takes the form
\[
\diff F(\omega) \left(\partial_t u + u\cdot\nabla u + \nabla p \right)=0.
\]
Since $F$ was arbitrary, and because our choice of spaces enforces the incompressibility constraint,
this yields the equations
\eqref{1}, \eqref{2} modulo the scaling factor which 
we will see later arises from a natural rescaling of the Hamiltonian.

In the following, we try to convince the reader that binormal curvature flow naturally arises as the motion of vortex filaments if the vorticity concentrates along a curve. In this case, the vorticity field is tangential to such a curve and the manifold of vorticity fields degenerates to
\[
\M = \left\{\mbox{ oriented, closed curves $\Gamma$ in }\R^3\right\}.
\]
Infinitesimal variations can be represented by vector fields on these curves, and thus, we identify 
\[
\T_{\Gamma}\M = \left\{\mbox{vector fields $\varphi$ on }\Gamma\right\}.
\]
These vector fields are the velocity fields which transport the curve, and thus, the $L^2$ metric tensor becomes
\[
\g_{\Gamma}(\varphi,\psi) = \int_{\Gamma}\varphi\cdot \psi\, d\Ha^1.
\]
 Likewise, the limiting skew-symmetric bilinear form reads
\[
\s_{\Gamma}(\varphi,\psi) =  \int _{\Gamma} \varphi\times\psi\cdot \tau_{\Gamma}\, d\Ha^1,
\]
where $\tau_{\Gamma}$ denotes the positively oriented unit  tangent on $\Gamma$. Finally, if the vorticity sharply concentrates around $\Gamma$ and vorticity is the only source for the kinetic energy, the Hamiltonian becomes the length of the curve, $H(\Gamma) = |\Gamma|$.

We now compute the Hamiltonian-Poisson system generated by these ingredients. If $\gamma: [0,L)\to \R^3$ is an arc-length parametrization of $\Gamma$, then
$
\diff H(\Gamma)(\varphi )= -\int_0^L \gamma''(s)\cdot \varphi(s)\, ds
$,
and thus $\gradient H(\Gamma) = -\gamma''$. It thus follows that 
\[
\{F,H\}(\Gamma) = - \int_0^L \left(\gradient F\times \gamma''\right)\cdot \gamma'\, ds =  \int_0^L \left(\gamma'\times \gamma''\right)\cdot \gradient F\, ds .
\]
Hence, using similar arguments as above, the Hamilton-Poisson system \eqref{100} becomes
\[
\diff F(\Gamma) \left(\partial_t \gamma - \gamma'\times \gamma''\right) = 0.
\]
Since $F$ was arbitrary, this shows that the curve evolves by binormal curvature flow.


We summarize the above discussion in the following table, where we introduce
subscripts  ${}_E$ and ${}_B$ to distinguish between (some) objects
that relate to the Euler equation and the binormal curvature flow, respectively.

\begin{table}[htp]
\small 
\begin{center}
\begin{tabular}{@\quad l@\quad|@\quad l@\quad}
Euler & BCF\\
\hline

$\displaystyle\M_E = \nabla\times \left\{ u:\: \nabla\cdot u=0\right\} $
&$\displaystyle \M_B = \left\{ \mbox{oriented curves $\Gamma$ in $\R^3$} \right\}$\\

$\displaystyle T_\omega \M_E =  \left\{ v:\: \nabla\cdot v=0\right\}$
&$\displaystyle T_\Gamma \M_B = \left\{ \mbox{normal vector fields along $\Gamma$} \right\}$ \\

$\displaystyle g_\omega(v,w) = \int v\cdot w \ dx$ &
$\displaystyle g_\Gamma(v,w) = \int_{\Gamma} v\cdot w \, d\calH^1 $
\\

$\displaystyle H_E(\omega) = \frac12\int |(\curl)^{-1}\omega|^2\, dx$
\qquad\qquad
&
$\displaystyle H_B(\Gamma) = |\Gamma|$
\\

$\displaystyle s_\omega(v,w) = \int (v\times w)\cdot \omega \, dx$ &
$\displaystyle s_\Gamma(v,w) = \int_{\Gamma}(v\times w)\cdot \tau\, d\calH^1 $
\\
$\displaystyle \{ F,G\}(\omega) = s_\omega( \gradient F,  \gradient G )$
\qquad\qquad
&
$\displaystyle \{ F,G\}(\Gamma) = s_\Gamma( \gradient F,  \gradient G )$
\end{tabular}
\end{center}
\label{default}
\end{table}%

We wish to make the above heuristic argument a little 
more convincing by comparing 
the above Hamiltonian-Poisson structures,
including the Poisson brackets,
at some $\omega\in \M_E$ and $\Gamma\in \M_B$
such that $\omega$ is $\ep$-concentrated around $\Gamma$.
In this situation, it turns out to that $H_E(\omega) \ge \frac {|\log\ep|}{4\pi}H_B(\Gamma) -  \mathcal{O}(1)$;
this is a consequence of Theorem \ref{T1} above.
It is therefore natural to
rescale the Euler Hamiltonian by introducing
\[
H_E^\ep (\omega) = \frac {4\pi}{|\log\ep|} H_E(\omega).
\]
(Note that the 
rescaling is exactly the one we obtain by time rescaling.) To compare the Poisson brackets, we need to be able to
(approximately) identify functions $F,G$
on $\M_E$ with 
functions on $\M_B$.
It is not clear how to do this in general, but  for certain functions there are
natural choices. 
In particular:
\begin{itemize}
\item We would like the rescaled Euler Hamiltonian $H^\ep_E$ to correspond to the
BCF Hamiltonian $H_B$.

\item Given $\phi \in C^2_c(\R^3;\R^3)$, we can define linear maps $\ell_\phi$
on the two manifolds which are natural counterparts of each other:
\[
\ell_{E,\phi}(\omega) := \int \phi \cdot \omega \ dx \ ,
\quad
\qquad
\ell_{B,\phi}(\Gamma) := \int_\Gamma \phi \cdot \tau \ d\calH^1 \ .
\]
\end{itemize}
We can then supplement the above table with a couple of new entries. It is straightforward to
check the following, and the proofs are essentially contained in \eqref{8} and \eqref{7} above.

\begin{table}[htp]
\small 
\begin{center}
\begin{tabular}{@\quad l@\quad|@\quad l@\quad}
Euler &BCF\\
\hline
$\displaystyle \gradient \ell_{\phi} = \nabla \times \phi \hspace{5em}\ $&$\displaystyle 
\gradient \ell_{\phi} = \tau \times (\nabla\times \phi)$ \\

$\displaystyle \{ H^\ep_E, \ell_{\phi} \}(\omega)$
&$\displaystyle \{ H_B, \ell_{\phi} \}(\Gamma)$\\
$\displaystyle\qquad
= \frac{4\pi}{|\log\ep|} \int \nabla(\nabla\times\phi) :  u\otimes u\, dx$
&$\displaystyle \qquad
=
 \int_\Gamma \nabla(\nabla\times\phi) : (I-\tau \otimes\tau) \ d\calH^1$
\end{tabular}
\end{center}
\label{default2}
\end{table}%

In this context, our second main result, Theorem \ref{T1}, 
strengthens and refines  the heuristic arguments described above by
providing a precise, quantitative version of the statement that
if $\omega \in \M_E$
is concentrated 
in an $\ep$
neighborhood of a curve $\Gamma\in \M_B$
and if $H_E^\ep(\omega)$ is close to $H_B(\Gamma)$, then
$\{ H^\ep_E, \ell_{\phi} \}(\omega)$ is close to
$\{ H_B, \ell_{\phi} \}(\Gamma)$.

\section{Proofs}\label{S:Proofs}

The remainder of the paper is devoted to the rigorous justification of Theorems \ref{T2} and \ref{T1}. 
We start with some preliminaries on geometric properties of non-self-intersecting $C^2$ curves and derive some properties of the flat norm.

%
%
%
%
%

%
%
%

\subsection{Curves for which $\|\kappa^*\|_{L^{1,\infty}}<\infty$.}

Assume that $\Gamma$ is a closed Lipschitz curve  such that
\begin{equation}
|\Gamma|=1, \qquad \mbox{ and } \quad \|\kappa^*_\Gamma \|_{L^{1,\infty}}<\infty.
\label{Gamma1}\end{equation}
In particular, this implies that $r_\gamma(s)>0$ a.e.  for any arclength parametrization $\gamma$. We will use the notation
\[
\mathcal T := \{ \gamma(s) + v :  s\in \R/\Z,  v\in \R^3, v\cdot \gamma'(s)=0, |v|<\frac 14 r(s)\} .
\]
Thus $\mathcal T$ is a tube of varying thickness around $\Gamma$. 
In general, $\mathcal T$ need not be open, but it is straightforward to 
check that it is measurable.
We  record some of its properties:

\begin{lemma}
For every $x\in \mathcal T$, there exists a unique closest point $P(x)$ in $\Gamma$, characterized by
$\dist (x,\Gamma) = |x- P(x)|$, and we may  define
$\zeta:\mathcal T\to \R/\Z$ by requiring that
\[
P(x) = \gamma(\zeta(x)) \qquad\mbox{ for all }x\in \mathcal T\, ,
\]
If $\gamma$ is $C^2$ then $\zeta$ is $C^1$ in $\mathcal T$, with
\begin{equation}
\grad \zeta (x) = \frac{\gamma'(\zeta(x))}{1 - \left(x-\gamma(\zeta(x))\right)\cdot \gamma''(\zeta(x))}.
\label{nablazeta}\end{equation}
More generally, if we only assume \eqref{Gamma1}, then
$\zeta$ is locally Lipschitz in $\mathcal T$, and 
\begin{equation}
\left | \frac 1 {|\nabla \zeta(x)|} - 1 \right| 
\le \frac {\dist(x,\Gamma)}{r(\zeta(x))} \qquad \mbox{a.e. }  x\in \mathcal T.
\label{loclip}\end{equation}
\label{L9cis}\end{lemma}

A similar function $\zeta$, but for smoother $\Gamma$ and defined on a tube of uniform diameter, can be found in \cite[Prop.\ 4]{JerrardSmets15}.

\begin{proof}
Fix $x\in \mathcal T$; then there exist $s\in \R/\Z$ and $v\in\R^3$
such that
\begin{equation}
x = \gamma(s) + v, \qquad v\cdot \gamma'(s) = 0, \qquad |v|< \frac 14 r(s).
\label{zeta1}\end{equation}
We may assume, by changing variables and reparametrizing $\gamma$, that $s=0$,
$\gamma(0) = 0\in \R^3$ and  $\gamma'(0) = (0,0,1)$, and we will write
\[
\gamma = (\gamma_\perp, \gamma_\parallel) \in \R^2\times \R,
\qquad  x = v= (v_\perp, 0), \qquad r_0  = r(0).
\] 
We will show that $P(x)$ is well-defined by proving that 
$\gamma(0)$ is the unique closest point in $\Gamma$ to $x$,
or in other words that 
\begin{equation}
\mbox{ if $0<|h|\le 1/2$, \ \  then $|x-\gamma(h)| > |v| = |v_\perp|$.} 
\label{closestpoint}\end{equation}
If $|h|\ge r_0$, then the definition of $r(0)$ implies that $|\gamma(h)|
= |\gamma(h) - \gamma(0)| \ge \frac 12 r_0$, so
by \eqref{zeta1},
\begin{equation}
|\gamma(h) - x| \ge |\gamma(h)| - |v| \ge \frac 14 r_0 > |v|.
\label{ghmx}\end{equation}
We thus assume $|h| <r_0$. By the definition of $r(\cdot)$,
\begin{equation}
\gamma_\parallel'(h)  = \gamma'(h)\cdot \gamma'(0) =  1- \frac 12 |\gamma'(h) - \gamma'(0)|^2
\ge 1 - \frac{h^2}{ 2 r_0^2} 
\label{gprime1}\end{equation}
for {\em a.e. }$|h|\le r_0$.  It follows that 
\begin{equation}
|\gamma_\perp' (h)|  = (1 - |\gamma_\parallel ' |^2)^{1/2}  \le  \frac{|h|}{r_0} .
\label{gprime2}\end{equation}
For $0<h<r_0$, we integrate these inequalities to obtain
\begin{equation}
h\ge \gamma_\parallel(h) \ge h(1 - \frac {h^2}{6 r_0^2}), \qquad |\gamma_\perp(h)|\le  \frac{h^2}{2r_0}.
\label{whereisgamma}\end{equation}
Thus
\begin{equation}
|\gamma(h) - x|^2 
\ge  \begin{cases}
h^2(1 - \frac {h^2}{6 r_0^2})^2 + (|v_\perp| - \frac {h^2}{2r_0})^2 
\quad&\mbox{ if } \frac {h^2}{2r_0} \le |v_\perp| , 
 \\
h^2(1 - \frac {h^2}{6 r_0^2})^2  
&\mbox{ if  not} .
\end{cases}
\label{distxgamma}\end{equation}
It follows  from this by elementary calculations, using the fact that $|v_\perp|< \frac 14r_0$,  that $|\gamma(h) - x|^2 > |v_\perp|^2$. The same estimate holds for $-r_0<h<0$, by essentially the same argument,
completing the proof of \eqref{closestpoint}.

The uniqueness of the closest point implies that $P(\cdot)$ and $\zeta(\cdot)$ are continuous in $\mathcal T$.
Indeed,  if $x_k\rightarrow x$ in $\mathcal T$ and $P(x_k)\to y$, then it is clear that $y\in \Gamma$ and that
$\dist(x,y) =  \lim_k \dist(x_k, \Gamma)  = \dist(x, \Gamma)$,
and hence that $y=P(x)$.

To verify that $\zeta$ is Lipschitz, consider $y\in \mathcal T$ near $x$, and let $h := \zeta(y) $.
If $y$ is sufficiently close to $x$, then
$h$ satisfies 
\[
|h| <r_0, \qquad y = \gamma(h) + w, \qquad w\cdot \gamma'(h) = 0.
\]
We would like to estimate $|h| = |\zeta(x)-\zeta(y)|$ in terms of $|x-y|$.
To do this,
note that
$|y-x| \ge |\gamma'(h)\cdot(y-x)| = |\gamma'(h)\cdot(\gamma(h)-x)|$.
By writing
\[
\gamma'\cdot (\gamma-x )
= (\gamma_\perp', \gamma_\parallel') \cdot (\gamma_\perp-v_\perp, \gamma_\parallel)
\]
and using the estimates for various components of $\gamma$ and $\gamma'$ in 
\eqref{gprime1}, \eqref{gprime2}, \eqref{whereisgamma}, we find that
\begin{equation}
 |h|( 1 - \frac {|v_\perp|}{r_0} -  \frac 76 \frac {h^2}{r_0^2} )
 \le |x-y| 
\label{lipest}\end{equation}
as long as $|h|<r_0$.
Since $|v_\perp|< \frac 14 r_0$, it follows that 
$|x-y| \ge \frac 12 |h|$ as long as $|h| < \frac 14 r_0$.
Because $x$ was arbitrary, this shows that for $x,y\in \mathcal T$,
\[
|x-y| \ge \frac  12 |\zeta(x)-\zeta(y)| \mbox{ as long as } |\zeta(x)-\zeta(y)|\le 
\frac 14 r(\zeta(x)) .
\]
It follows that $\zeta$ is locally Lipschitz in $\mathcal T$
and hence {\em a.e.} differentiable in $\mathcal T$. 

Next, at every point of differentiability, it follows from \eqref{lipest} that
\[
\frac 1{|\nabla\zeta(x)|}  \ = \ \liminf_{y\to x} \frac{|x-y|}{|\zeta(x) - \zeta(y)|} 
\ge  1 -  \frac{ |v_\perp|}{r_0}  = 1 - \frac{\dist(x,\Gamma)}{r(\zeta(x))}.
\]
On the other hand, if $\zeta$  is differentiable at $x$ and
$\mathcal T$ has density $1$ at $x$ in the sense that
$\lim_{r\searrow 0}\frac{|B_r(x)\cap \mathcal T|}{|B_r(x)|} =1$, then there exists a sequence $y_k\to x$ 
such that, writing $h_k := \zeta(y_k)$ and still assuming $\zeta(x)=0$,
\begin{align*}
1 = \lim_{k\to \infty}\frac{ \gamma'(0) \cdot (y_k-x)}{|y_k-x|} 
&= \lim_{k\to \infty}
\frac{ \gamma'(h_k) \cdot (y_k-x)}{|y_k-x|} \\
&= \lim_{k\to \infty}
\frac{ \gamma'(h_k) \cdot (\gamma(h_k)-x)}{|y_k-x|}.
\end{align*}
Thus, again decomposing $\gamma$ and $\gamma'$ and using 
\eqref{gprime1}, \eqref{gprime2}, \eqref{whereisgamma}, 
we check that
\[
1 
\le \lim_{k\to \infty} \frac{ \gamma_{\parallel}'(h_k) + |\gamma_\perp'(h_k)| \ (|\gamma_\perp(h_k)| + |v_\perp|)}{|y_k-x|}
= 
\lim_{k\to \infty} \frac {h_k}{|y_k-x|} \left( 1 + \frac {|v_\perp|}{r_0}\right)
\]
which implies that $|\nabla \zeta(x)|^{-1} \le 1+\frac{|v_\perp|}{r_0}= 1 + \frac{\dist(x,\Gamma)}{r(\zeta(x)}$. Thus we have proved \eqref{loclip}.

Finally, if $\gamma$ is $C^2$, then it follows from what we have said above 
(which shows that $\zeta(x) = s$ when \eqref{zeta1} holds) that
\[
\varphi(x, \zeta(x)) = 0 \qquad \mbox{ for }
\varphi(x,t) = (x-\gamma(t))\cdot \gamma'(t).
\]
Because $|\gamma''(s)|\le \frac 1 {r(s)}$, the definition of $\mathcal T$ implies that
$\partial_t \varphi(x,t) \le - 3/4$ everywhere in $\mathcal T$.
Thus the implicit function theorem implies that $\zeta$ is $C^1$ and that
\[
\gamma'_i(\zeta(x)) - \left[1 - \left(x-\gamma(\zeta(x))\right)\cdot \gamma''(\zeta(x))\right]\partial_i\zeta(x) = 0.
\]
We obtain \eqref{nablazeta} by rearranging.

\end{proof}

%

%
%
%
%

The following Lemma is an important reason for the relevance of the weak $L^1$ norm
of $\kappa^*$.

\begin{lemma}
There exists an absolute constant $C$ such that for every $x\in \R^3$ and every $r>0$,
\begin{equation}
\left| \{ s\in \R/\Z : |\gamma(s) - x| < r \} \right| \le C  r \| \kappa^*\|_{L^{1,\infty}} .
\label{linear.est}\end{equation}
\label{L.linear}\end{lemma}

In fact the proof will show that we may take $C=8$.

\begin{proof}
Fix $x\in \R^3$ and $r>0$, and let $A :=\{ s\in \R/\Z : |\gamma(s) - x| < r \} $.
We consider 2 cases.

{\bf Case 1}: $|A| \le 8r$.

It is clear from the definition of $r(\cdot)$ that $\frac 1{\kappa^*(s)} = r(s)\le 1$ for every $s$,
and hence that $| \{ s\in \R/\Z : \kappa^*(s) \ge 1 \}| \ge 1$, which implies that 
$\| \kappa^*\|_{L^{1,\infty}} \ge 1$. Thus \eqref{linear.est} holds.

{\bf Case 2}: $|A|>8r$.

For any $s\in A$, there must then exist $t\in A$ such that $4r < |s-t| <1/2$.
The definition of $A$ and the triangle inequality imply that 
$|\gamma(s) - \gamma(t)| < 2r$.
Then 
\[
2r > \left| \gamma'(s)\cdot (\gamma(t)-\gamma(s))\right| =   \left|\gamma'(s)\cdot \gamma'(\tau)\right| \ |t-s|
\]
for some $\tau$ between $s$ and $t$. It follows that $\gamma'(s)\cdot \gamma'(\tau) < \frac 12$, and
hence
that  $|\gamma'(s)-\gamma'(\tau)|^2 = 2 - 2\gamma'(s)\cdot \gamma'(\tau) > 1$.

These facts imply that $r(s)\le 4r$. To check this, we must show that
for every $\rho > 4r$, one of two inequalities appearing in the definition
\eqref{r.def} of $r(\cdot)$ is violated.
The first of these inequalities is
\[
\ |\gamma(s+h) -\gamma(s)|\ge  \frac \rho 2 \ \mbox{ for all }|h|\ge \rho,
\]
which is violated by $s+h = t$, if $4r <  \rho <|t-s|$.
The other inequality is 
\[
|\gamma'(s+h)-\gamma'(s) | \le \frac {|h|}\rho   \ \mbox{ for all }|h|\le \rho,
\]
and is violated by $s+h = \tau$, if $|\tau - s| <\rho$. Thus at least one
of these must fail when $\rho >4r$, so
$r(s)\le 4r$
as claimed.

Since $s$ was arbitrary, we conclude 
$A\subset \{ s : \kappa^*(s) \ge \frac 1{4r}\}$,
and hence that 
\[
|A| \le  4r \| \kappa^*\|_{L^{1,\infty}}.
\]
This proves Lemma \ref{L.linear}.
\end{proof}

\subsection{Some properties of the flat norm}\label{S:Flatnorm}

\begin{lemma}\label{L12}
Let $\omega$ be a  divergence-free vector-valued measure on $\R^3$ such that $\|\omega\|_F<\infty$. 
Then
\begin{equation}
\| \omega\|_F 
= \inf \{ |\varphi|(\R^3)  \ : \ \varphi \in \mathcal M(\R^3;\R^3), \ \nabla\times \varphi = \omega \}
\label{612}\end{equation}
and the infimum is attained.
\end{lemma}

In the statement of the lemma,  $\mathcal M(\R^3 ;\R^3)$ denotes the space of vector-valued measures on
$\R^3$, and 
$|\varphi|$ denotes the total variation of the measure $\varphi$, defined by
\[
|\varphi|(\R^3) =\sup\left\{ \int f\cdot d\varphi :\: f\in C_c(\R^3;\R^3),\,  \|f\|_{\Li}\le 1\right\}.
\]
Clearly, if $\varphi$ is absolutely continuous and integrable, then $|\varphi|(\R^3) = \|\varphi\|_{L^1}$. 

The lemma is a special case of variant of \cite[4.1.12]{Federer69}, rewritten in the language of vector calculus.
We provide a proof for the reader's convenience.

\begin{proof}
We first observe that for any $\xi\in C_c^{\infty}(\R^3;\R^3)$ with $\|\curl \xi\|_{C^0}\le1$ and any vector-valued Radon measure $\varphi$ such that $\nabla\times \varphi = \omega$,
\[
\int \xi\cdot \, d\omega = \int \curl\xi\cdot d\varphi\le |\varphi|(\R^3).
\]
It follows that
\[
\| \omega\|_F 
\le \inf \{ |\varphi|(\R^3)  \ : \ \varphi \in \mathcal M(\R^3;\R^3), \ \nabla\times \varphi = \omega \}.
\]
The other inequality is a nice application of the Hahn--Banach theorem. We denote by $X$ the vector space $C_c(\R^3;\R^3)$ equipped with the maximum norm, and define the linear subspace
\[
Y:=\left\{f\in X:\: \exists\xi \in C_c^{1}(\R^3;\R^3)\mbox{ with } f=\curl \xi\right\},
\]
also equipped with the maximum norm. We furthermore define implicitly
\[
L(\curl \xi):= \int \xi\cdot d\, \omega
\]
for vector fields $\nabla \times \xi\in Y$.
This linear operator is well-defined thanks to the divergence-free condition for $\omega$. Indeed, if $\xi_1$ and $\xi_2$ are such that $\curl\xi_1=\curl \xi_2$, then $\xi_1 - \xi_2 = \grad\rho$ for some $C^2_c$ function $\rho$, and thus
\[
\int \xi_1\cdot d\omega - \int \xi_2 \cdot d \omega = \int \grad\rho\cdot d \omega = 0,
\]
because $\div \omega=0$. Moreover, $L: Y\to \R$ is bounded with operator-norm
\[
\|L\|_{Y\to\R} =\sup\left\{L(f):\: f\in Y\mbox{ with }\|f\|_{\Li}\le 1\right\}= \|\omega\|_F.
\]
By the Hahn--Banach theorem (see, e.g., \cite[Theorem 3.3]{Rudin}), there exists a linear function $\bar L:X\to \R$ such that $\bar L$ agrees with $L$ on $Y$, and whose norm is not larger than that of $L$:
\[
\|\bar L\|_{X\to \R} = \|L\|_{Y\to\R} = \|\omega\|_F.
\]
By duality (see, e.g., \cite[Chapter 1.8]{EvansGariepy92}), there exists a vector-valued Radon measure $\varphi$ on $\R^3$ such that
\[
\bar L(f) = \int f\cdot d\varphi\quad\mbox{for all }f\in C_c(\R^3;\R^3).
\]
It follows that
\[
\int \xi\cdot d \omega = L(\curl \xi) = \int \curl \xi\cdot d\varphi
\]
for all $\xi \in C_c^{1}(\R^3;\R^3)$, that is, $\omega = \nabla\times \varphi$. Furthermore,
\[
\|\omega\|_F =  \|\bar L\|_{X\to \R} = \sup\left\{ \int f\cdot d\varphi:\: f\in C_c(\R^3:\R^3),\, \|f\|_{\Li}\le 1\right\} = |\varphi|(\R^3).
\]
 This completes the proof of \eqref{612}, and also shows that  the infimum is attained.
\end{proof}

\begin{lemma}\label{L12bis}
 Assume that $w\in L^2(\R^3;\R^3)$ and that $\nabla \cdot w = 0$.
If $\| \nabla\times w\|_F <\infty$, then  
\[
\| w \|_{L^{1,\infty}} \le \| \nabla\times w\|_F.
\]
\end{lemma}

\begin{proof}
By Lemma \ref{L12}, there exists a vector-valued measure $\varphi$ on $\R^3$ such that $\curl \varphi =\curl w$ and $|\varphi|(\R^3) = \|\curl w\|_F$. The linear relation between $\varphi$ and $w$ is retained by convolution, and thus, denoting by $\varphi^{\eps}$ and $w^{\eps}$ the convolutions of $\varphi$ and $w$, respectively, with a smooth mollification kernel $\re$, on the Fourier level it holds
\[
\widehat{w^{\eps} } (\xi) = m(\xi) \widehat{\varphi^{\eps}}(\xi)\quad\mbox{where }m(\xi)= \I - \frac{\xi\otimes\xi}{|\xi|^2}, 
\]
as can be derived from the elementary formula $\curl\curl A= -\laplace A+ \grad\left(\div A\right)$. Because $m$ is homogeneous of degree zero, bounded, and smooth away from the origin, we infer from H\"ormander's multiplier theorem \cite[Theorem 0.2.6]{Sogge} that
\begin{equation}\label{62bis}
\|w^{\eps}\|_{L^{1,\infty}} \lesssim \|\varphi^{\eps}\|_{L^1}.
\end{equation}
Clearly,
\[
\|\varphi^{\eps}\|_{L^1} = \sup_{\|\zeta\|_{\Li}\le1} \int\zeta\cdot \varphi^{\eps}\, dx= \sup_{\|\zeta\|_{\Li}\le1} \int \zeta\ast\rho^{\eps}\cdot d\varphi \le |\varphi|(\R^3),
\]
since $\|\zeta\ast\rho^{\eps}\|_{\Li}\le \|\zeta\|_{\Li}$, and thus, passing to the limit in \eqref{62bis} yields
\[
\|w\|_{L^{1,\infty}} \lesssim |\varphi|(\R^3).
\]
To conclude, it remains to combine this estimate with the statement of Lemma \ref{L12}.
\end{proof}

\begin{lemma}\label{L6}
Let $\mu\in \M(\R^3;\R^3)$ be  compactly supported  and divergence-free. 
If $\re$ is a nonnegative function
such that $\mbox{supp}(\re)\subset B_\eps(0)$ and $\int \re = 1$, then
\[
\| \rho^{\eps} \ast \mu - \mu\|_F \le \ep |\mu|(\R^3).
\]
\end{lemma}

\begin{proof}
For any $z\in \R^3$, let us write $\sigma_z \mu$ to denote the measure
defined by
\[
\int \phi\cdot d(\sigma_z \mu) :=
\int \phi(\tacka+z) \cdot d\mu 
=
\int_\Gamma \phi(x+z) \cdot \tau_\Gamma(x)\, d\calH^1(x).
\]
We also define a vector-valued measure $R_z$ by
\[
\int \phi \cdot dR_z = \int_\Gamma \int_0^1 \phi(x +sz) \cdot (z \times \tau_\Gamma(x))\, ds d\calH^1(x).
\]
It is a standard fact that
\begin{equation}
\nabla \times R_z = \sigma_z \mu - \mu.
\label{homotopy}\end{equation}
We recall the proof:  for any $\phi \in C^\infty_c(\R^3;\R^3)$,
\begin{align*}
\int \phi \cdot  d(\nabla \times R_z) 
&= 
\int \nabla \times \phi \cdot  dR_z \\
&=
\int_\Gamma \int_0^1  \left(\nabla \times\phi\right)(x +sz) \cdot \left(z \times \tau_\Gamma(x) \right)\, ds d \calH^1(x) .
\end{align*}
Straightforward computations show that
\[
\left(\nabla\times \phi\right) (x+sz) \cdot \left(z\times \tau_\Gamma\right)
= 
\sum_j \frac{\partial}{\partial s}\left( \phi_j(x+sz) \right)\tau_{\Gamma,j} - \sum_j \tau_\Gamma\cdot \nabla \phi_j (x+sz)z_j.
\]
Clearly,
\[
\int_0^1 \int_\Gamma  \tau_\Gamma \cdot\nabla\phi_j(x +sz) z_j\, d\calH^1(x) ds
= 0
\]
since the integral over $\Gamma$ vanishes for every $s$. Thus we conclude from the fundamental
theorem of calculus that
\begin{align*}
\int \phi \cdot  d\left(\nabla \times R_z \right)
&
=
\sum_j\int_\Gamma \int_0^1\frac{d}{ds} \phi_j(x+sz) \tau_{\Gamma,j}(x)\, ds   d\calH^1(x) \\
&=
\int_\Gamma \left( \phi(x+z) - \phi(x) \right)\cdot \tau_{\Gamma}(x) \,  d\calH^1(x),
\end{align*}
which proves \eqref{homotopy}.
It follows that
\[
 \rho^\ep * \mu - \mu
= \int \rho^\ep(z) ( \sigma_z\mu - \mu)\, dz
= \nabla \times \int \rho^\ep(z) R_z \ dz
\]
in the sense of distributions.  Hence by Lemma
\ref{L12},
\begin{align*}
\|\rho^{\eps}\ast\mu - \mu  
\|_F \ \le \ 
\left|   \int \rho_\ep(z) R_z \ dz \right|(\R^3)
&\ \le \ 
 \int \rho_\ep(z) \left| R_z \right|(\R^3) \ dz
\\
& \ \le \ \sup_{|z|\le \ep} \left| R_z \right|(\R^3).
\end{align*}
However, it is easy to check from the definition that
$ \left| R_z \right|(\R^3) \le |z| \ |\Gamma|$ for every $z$,
so the conclusion follows.
\end{proof}

\subsection{Proof of Theorem \ref{T1}}\label{Proof.Theorem2}

\begin{proof}[Proof of Theorem \ref{T1}.] 
The fact that $0\le \exc_{\eps} + \mathcal{O}\left(\|\kappa^*\|_{L^{1,\infty}}^2|\log\eps|^{-1}\right)$  follows immediately from the main estimate \eqref{main.estimate}.

We prove \eqref{main.estimate}  componentwise. For this purpose, we fix $i,j\in\{1,2,3\}$ and write
\begin{eqnarray*}
\int \phi \ue_i\ue_j\, dx &=& \int \phi \ve_i\ve_j\, dx\\
&&\mbox{} + \int \phi \left(\ue_i-\ve_i\right)\left(\ue_j-\ve_j\right)\, dx\\
&&\mbox{} + \int \phi\left(\ue_i - \ve_i\right)\ve_j\, dx + \int\phi \left(\ue_j - \ve_j\right)\ve_i\, dx.
\end{eqnarray*}
Thanks to \eqref{ve.4} in Proposition \ref{Prop.ve}
(whose proof appears in Section \ref{S:prototype} below),  
conclusion \eqref{main.estimate} follows from the estimates
\begin{eqnarray}
\int \phi \left(\ue_i-\ve_i\right)\left(\ue_j-\ve_j\right)\, dx & =& \frac{|\log\eps|}{2\pi}\|\phi\|_{L^{\infty}} \exc_{\eps} +
\  \mathcal{O}(\|\phi\|_{L^{\infty}} \| \kappa^*\|_{L^{1,\infty}}^2),\label{50}\\
\int \phi\left(\ue_i - \ve_i\right)\ve_j\, dx& \le &C|\log\eps|\|\phi\|_{L^{\infty}} \exc_{\eps} + \  
 \mathcal{O}(\|\phi\|_{L^{\infty}} \| \kappa^*\|_{L^{1,\infty}}^2)
\label{51},
\end{eqnarray}
for all $i,j\in\{1,2,3\}$ and all $\phi\in W^{1,\infty}(\R^3)$.

To prove \eqref{50} and \eqref{51}, we need some preparation. 
First, note from assumption \eqref{7b} and estimate \eqref{ve.0} in Proposition \ref{Prop.ve} 
that
\begin{equation}
\| \nabla\times \ue - \nabla\times \ve\|_F \le 2 \ep.
\label{137}\end{equation}
Next, we write
$\ve$ in the form $\ve = \nabla \times \Phi^\ep$, for $\Phi^\ep = \G \ast \re \ast \mu_\Gamma$, where $\G$ denotes the Newtonian potential $\G(z) = \frac1{4\pi |z|}$.
 We recall that $\Phi^{\eps}(x) =( \G\ast\mu_{\Gamma})(x)$ for all $x$ such that $\dist(x,\Gamma)\ge \eps$ by the mean value property of harmonic functions. Notice also that $\int d\mu_{\Gamma} = 0 $, because $\Gamma$ is a closed curve. As a consequence,
\[
|\Phi^\ep(x) | = \left| \int  (\G(x-z) - \G(x))\, d\mu_\Gamma (z)\right| \lesssim \int \frac{|z|}{|x|^2}\, d|\mu_\Gamma| \lesssim \frac{1}{|x|^2},
\]
whenever $|x|$ is sufficiently large.
Now let $\chi\in C^\infty_c(B_2(0))$ be a function such that $\chi = 1$ in $B_1(0)$,
and let $\chi_\lambda(x) = \chi(x/\lambda)$.
Then using the above decay of $\Phi^\ep$, one easily checks that
\begin{eqnarray*}
\int \ve\cdot \left(\ve-\ue\right)\, dx 
&=&
\lim_{\lambda\to \infty}\int (\chi_\lambda \Phi^\ep)\cdot \curl (\ve - \ue)\, dx \\
&\le&
\liminf_{\lambda
\to\infty}
\|\nabla\times (\chi_\lambda \Phi^\ep) \|_{L^\infty} \|\curl \ve - \curl \ue\|_F.
\end{eqnarray*}
Also,  after again using the above decay of $\Phi^\ep$, we find 
\[
\limsup_{\lambda\to \infty}  \| \nabla \times (\chi_\lambda \Phi^\ep) \|_{L^\infty}   = \| \ve\|_{L^\infty}.
\]
Thanks to  \eqref{ve.1} (with $q=\infty$) in Proposition \ref{Prop.ve} and  \eqref{137}
we conclude that
\[
\left|\int \ve\cdot \left(\ve-\ue\right)\, dx\right| \lesssim 
 \| \kappa^*\|_{L^{1,\infty}}.
\]
It thus follows by \eqref{ve.2} in Proposition \ref{Prop.ve} and the definition \eqref{exc.def}  of $\exc_\ep$ that 
\begin{eqnarray}
\int |\ue-\ve|^2\, dx &=& \int |\ue|^2 \, dx - \int |\ve|^2\, dx + 2\int \ve\cdot (\ve-\ue)\, dx\nonumber\\
&=& \frac{|\log\eps|}{2\pi} \exc_{\eps} +  \mathcal{O}\left( \| \kappa^*\|_{L^{1,\infty}}^2\right).\label{138}
\end{eqnarray}
(Recall that $\kappa^*\ge 1$.) From this estimate, we readily deduce \eqref{50}.

We turn to the proof of \eqref{51}. We let $1<p<2<q<\infty$ be arbitrarily fixed such that $1=\frac1p + \frac1q$. Then by H\"older's inequality
\[
\left|\int \phi\left(\ue_i - \ve_i\right)\ve_j\, dx\right|\le \|\phi\|_{L^{\infty}} \|\ue-\ve\|_{L^p} \|\ve\|_{L^q}.
\]
We invoke the interpolation inequality $\|f\|_{L^p} \lesssim \|f\|_{L^{1,\infty}}^{\frac{2-p}p} \|f\|_{L^2}^{\frac{2p-2}p}$ (the short proof of which can be found in the appendix).
We also have
\[
\| \ue-\ve\|_{L^{1,\infty}} \lesssim \|\nabla\times \ue - \nabla\times \ve\|_F \lesssim \eps
\]
by Lemma \ref{L12bis} and \eqref{137}.
Combining this with \eqref{ve.1} in Proposition \ref{Prop.ve}, we find that 
\[
\left|\int \phi\left(\ue_i - \ve_i\right)\ve_j\, dx\right|\lesssim \|\phi\|_{L^{\infty}}  \|\ue-\ve\|_{L^2}^{\frac{2p-2}p} \|\kappa^*\|_{L^{1,\infty}}.
\]
Now we choose $p=4/3$ and apply Young's inequality $ab\le  a^{2}/2+ b^{2}/2$ together with  \eqref{137} and \eqref{138} to deduce \eqref{51}.
This proves Theorem \ref{T1}.
\end{proof}

\subsection{Proof of Theorem \ref{T2}}

\begin{proof}[Proof of Theorem \ref{T2}]

Throughout this proof, implicit constants hidden in symbols such as 
$\lesssim$ or $\mathcal{O}(\cdots)$ may depend on 
quantities appearing in assumptions \eqref{Rgamma}, \eqref{initial.ue},
and \eqref{kstar.uniform}, but are independent of $\ep$ and of properties
of $\Lambda, \Gamma$ and $\ue$ not appearing in the assumptions. 

Let us define
\begin{gather*}
R_\gamma := \frac 14 \inf \left\{   r_{\gamma(t,\cdot)} (s)  \ : \ 0\le t\le T,s\in \R/\Z \right\},
\\
N(\Gamma,t) :=  \left\{ x\in \R^3 : \dist(x, \Gamma(t)) < R_\gamma\right\}.
\end{gather*}
We have assumed that $R_\gamma>0$, see \eqref{Rgamma}.
For every $t\in [0,T]$, according Lemma \ref{L9cis},
there is a (well-defined) map $\zeta_t : N(\Gamma,t)\to \R/\Z$ 
characterized by
\[
\qquad |x-\gamma(t, \zeta_t(x)) | = \dist(x, \Gamma(t)).
\]
Recall from \eqref{kappastar}, \eqref{ksgpp} that $r_{\gamma(t)}(s) \le |\partial_{ss}\gamma(t,s)|^{-1}$
for all $t,s$. Thus the definitions 
entail that 
\[
|x - \gamma(t,\zeta_t(x))| |\partial_{ss}\gamma(t,\zeta_t(x))|\le 1/4 \qquad\mbox{ in }N(\Gamma,t)
\]
 so that $\|\grad \zeta_t\|_{W^{2,\infty}}\lesssim1$ thanks to assumption \eqref{Rgamma} and \eqref{nablazeta} in Lemma \ref{L9cis}. 
Similar estimates hold for the temporal derivatives of $\zeta_t$ and $\grad\zeta_t$. Indeed, differentiating the defining condition $(x-\gamma(t,\zeta_t(x))\cdot \partial_s\gamma (t,\zeta_t(x))=0$ with respect to $t$ and recalling that $\gamma$ is a solution to the binormal curvature flow, we compute the identity
\[
\partial_t \zeta_t(x) = \frac{(x-\gamma(t,\zeta_t(x))\cdot(\partial_s\gamma(t,\zeta_t(x))\times \partial_s^3\gamma(t,\zeta_t(x)))}{1-(x-\gamma(t,\zeta_t(x))\cdot\partial_s^2\gamma(t,\zeta_t(x))}.
\]
In view of \eqref{Rgamma}, it is thus not difficult to infer $\|\partial_t \zeta_t\|_{W^{1,\infty}} \lesssim 1$.

Following \cite{JerrardSmets15}, we now define
\[
f(r^2) :=\begin{cases}
 \left(1 - (\frac {4r^2}{R_\gamma^2}) \right)^3 &\mbox{ if }r \le \frac 12 R_\gamma\\
 0&\mbox{ if not},
 \end{cases}
\] 
and for  $x\in \R^3$ and $0\le t\le T$, we define
\[
X_{\gamma(t)}(x) := 
\begin{cases}
f\left( \dist^2(x, \Gamma(t)) \right)\, \partial_s\gamma(t, \zeta_t(x)) \quad&\mbox{ if }x\in N(\Gamma,t)
\\
0&\mbox{ if not}.
\end{cases}
\]
We remark that as a result of \eqref{Rgamma}, the above bounds on $\zeta_t$, and because $\gamma$ is solution of the binormal curvature flow, we find that
\begin{equation}\label{Cgamma}
\sup_{t\in [0,T] }\| X_{\gamma(t)}(\cdot)\|_{W^{3,\infty}} \lesssim 1,
\qquad
\sup_{t\in [0,T] } \| \nabla \times \partial_t X_{\gamma(t)}\|_{L^{\infty}} \lesssim 1\, .
\end{equation}
In addition, the fact that $\gamma$ is a binormal curvature flow endows 
$X_\gamma$ with certain remarkable properties (see \eqref{rprop} above), established in
\cite{JerrardSmets15}, which will be recalled below.

We define
\[
E_\gamma(\Lambda, t) := 
1 -  \int   X_{\gamma(t)} \cdot d\mu_{\Lambda(t)}
= 
 \int_{\Lambda(t)} (1-   X_{\gamma(t)} \cdot \tau_{\Lambda(t)})\, d\calH^1
\]
and
\[
E_\gamma(\mu^\ep, t) := 
1 -  \int  X_{\gamma(t)} \cdot  d\mu^\ep(t) , \qquad \mu^\ep := \curl \ue(t,\cdot).
\]
Then for every $t$, by assumption \eqref{7b} and \eqref{Cgamma},
\begin{equation}\label{macro1}
\left|  E_\gamma(\mu^\ep,t)  - E_\gamma(\Lambda, t) \right|
= 
\left| \int X_{\gamma(t)} \cdot d(\mu^\ep(t) - \mu_{\Lambda(t)}) \right|
\lesssim \ep .
\end{equation}
Also, it follows from assumptions \eqref{7b} and \eqref{initial.ue} that $\| \mu_{\Gamma(0)} - \mu_{\Lambda(0)}\|_F \le 2\ep$, 
so 
\begin{equation}
\label{macro3}
|E_{\gamma}(\Lambda,0)| = \left| \int X_{\gamma(0)}\cdot d(\mu_{\Gamma(0)} - \mu_{\Lambda(0)})\right|\lesssim  |\log \eps|^{-1}.
\end{equation}
Moreover, (suppressing for readability the dependence on $t$ of various quantities)
it follows from \eqref{8} that
\[
\frac d{dt}
E_\gamma(\mu^\ep, t) 
=
-\int \partial_t X_\gamma \cdot  d\mu^{\eps}\
-  \frac {4\pi}{|\log\ep|} \int \nabla (  \nabla\times X_\gamma) : \ue \otimes \ue\, dx.
\]
From hypothesis \eqref{7b} and \eqref{Cgamma},
\[
\int\partial_t X_\gamma \cdot d\mu^{\eps} = 
\int \partial_t X_\gamma \cdot d\mu_{\Lambda} 
+  \mathcal{O}(\ep),
\]
and Theorem \ref{T1} and assumption \eqref{kstar.uniform} imply that
\[
\frac{4\pi}{|\log\ep|}\int\nabla(\nabla\times X_\gamma) : \ue\otimes \ue \, dx
=
\int_{\Lambda} \nabla(\nabla\times X_\gamma) : (I - \tau_{\Lambda} \otimes \tau_{\Lambda})d\calH^1  \ + \ 
\mathcal{O}\left(|\log\ep|^{-1}\right).
\]
Note also that for every vector field $\phi$,
\[
\nabla(\nabla\times\phi ) : I = 
\partial_i(\ep_{jkl}\partial_k \phi^l) \delta_{ij} =
\ep_{jkl}\partial_j \partial_k \phi^l = 0.
\]
Thus
\begin{equation}
\frac d{dt}
E_\gamma(\mu^\ep,t)
=
-\int_{\Lambda} \left[  \partial_t X_\gamma \cdot \tau_\Lambda
-
 \nabla(\nabla\times X_\gamma) :  \tau_\Lambda \otimes \tau_\Lambda \right] \ d\calH^1 
 + \mathcal{O}\left(|\log\ep|^{-1}\right) .
\label{macro2}\end{equation}
However, it is proved in \cite[Prop.\ 4]{JerrardSmets15} that for {\em any } unit vector
$\xi$,
\[ \left| \partial_t X_\gamma \cdot \xi
-
 \nabla(\nabla\times X_\gamma ): \xi \otimes \xi
 \right|
 \le
 K( 1 - X_\gamma \cdot \xi)
\]
where $K$ depends only on $R_\gamma$ and $\sup_{0\le t \le T}\|\partial_s^3 \gamma(t,\cdot)\|_{L^{\infty}}$.
This is the remarkable property mentioned above, reflecting the 
fact that $\gamma$ is a binormal curvature flow.
As a result,
\[ 
\left|\int_{\Lambda}  \partial_t X_\gamma \cdot \tau_\Lambda
-
 \nabla(\nabla\times X_\gamma) :  \tau_\Lambda \otimes \tau_\Lambda \ d\calH^1  \right|
 \le
K E_\gamma(\Lambda, t).
\] 
Combining this with \eqref{macro1}, \eqref{macro3} and \eqref{macro2},  we conclude that
\begin{align*}
E_\gamma(\Lambda, t) &\le \int_0^t \frac d{d\tilde t}
E_\gamma(\mu^\ep,\tilde t) \, d\tilde t \  + E_\gamma(\mu^\ep, 0) \ + \ \mathcal{O}(\ep)
\\
&\le
K \int_0^t E_\gamma(\Lambda, \tilde t)\ d\tilde t
+ \mathcal{O} \left(|\log\ep|^{-1}\right)
\end{align*}
for $0<t\le T$. It then follows from Gr\"onwall's inequality that
\begin{equation}\label{EGamma.est}
E_\gamma(\Lambda, t) \lesssim \frac {e^{Kt}} {|\log \ep|}  = \mathcal{O}\left(|\log\ep|^{-1}\right) \qquad \mbox{ for }0\le t \le T. 
\end{equation}

Finally, we show that $E_\gamma(\Lambda, t)$ controls a certain distance between
$\Gamma(t)$ and $\Lambda(t)$. We will suppress the variable $t$,
as it is not relevant here. Let $\lambda: [0,T]\times \R/Z\to \R^3 $ be an arc-length parametrization of $\Lambda$ having the same orientation as $\gamma$.
For $s\in \R/\Z$, let $\delta_\Gamma(s) := \dist(\lambda(s), \Gamma)$. 
We will show that
for all small enough $\ep$ and for every $t\in [0,T]$,
\begin{equation}\label{W11}
\sup_s\inf_{\sigma} |\lambda(s) - \gamma(\sigma) |^2 = \sup_s \delta^2_{\Gamma}(s) \lesssim E_\gamma(\Lambda)  \lesssim |\log\ep|^{-1}
\end{equation}
for all sufficiently small $\ep$. 
The proof of \eqref{W11} is essentially contained in \cite[Lemmas 4-5]{JerrardSmets13},
but we  recall the argument for the convenience of the reader.

First, it follows from the definition of $f$ that if $x =\lambda(s)$ for  any $s\in \R/\Z$, then
\begin{equation}\label{L2b0}
1 - X_\Gamma(x)\cdot \tau_\Lambda(x) \ge 1 - |X_\gamma(x)| = 1 - f(\delta^2_{\Gamma}(s)) \gtrsim \min\left\{1,  \delta_\Gamma^2(s)\right\}.
\end{equation}
Thus thanks to \eqref{EGamma.est}
\begin{equation}
\int_{\R/\Z} \min\{1, \delta_\Gamma^2(s) \}\,  ds  \ \lesssim E_\gamma(\Lambda) \lesssim |\log\ep|^{-1}.
\label{L2bound}\end{equation}
We now consider  $s\in \R/\Z$ such that $\delta_\Gamma(s)< R_\gamma$, and hence
$\zeta$ is well-defined near $x = \gamma(s)$. For such $s$,
we will write $\sigma(s) = \zeta(\lambda(s))$, so that $\gamma(\sigma(s)) = P(\lambda(s))$.
Note that if  $\gamma'(\sigma(s))\cdot \lambda'(s) \le 0$, then  $1- f(\delta_\Gamma^2)\gamma'(\sigma(s))\cdot \lambda'(s) \ge 1$,
and if not, then
\[
1- f(\delta_\Gamma^2)\gamma'(\sigma(s))\cdot \lambda'(s) \ge 1- \gamma'(\sigma(s))\cdot \lambda'(s)  = \frac 12|\gamma'(\sigma(s))-\lambda'(s)|^2.
\]
Either way, it follows that 
\begin{equation}
\label{L2low}
1 - X_\gamma(x)\cdot \tau_\Lambda (x) \ge \frac 14|\gamma'(\sigma(s))-\lambda'(s)|^2.
\end{equation}
Next, recalling that $(x - \gamma(\zeta(x))\cdot \gamma'(\zeta(x)) = 0$, we have
\begin{align*}
\frac 12 \frac d{ds}\delta_\Gamma^2(s)
&=
\big(\lambda(s) - \gamma(\sigma(s))\big) \cdot \big( \lambda'(s)  - \gamma'(\sigma(s))\sigma'(s) \big)\\
&=
\big(\lambda(s) - \gamma(\sigma(s))\big) \cdot \big( \lambda'(s) - \gamma'(\sigma(s))\big).
\end{align*}
Since $|a\cdot b|\le \frac 12( |a|^2 + |b|^2)$, we can combine this with \eqref{L2b0} to obtain
\begin{equation}
 \left| \frac d{ds} \delta^2_\Gamma(s) \right|
\lesssim  1 - X_\gamma(\lambda(s)) \cdot \lambda'(s)
\label{almostW11}\end{equation}
as long as $\delta_\Gamma(s) < R_\Gamma$.  If $I\subset \R/\Z$ is any interval on which
$\delta_\Gamma(s) < R_\gamma$, we can integrate \eqref{almostW11} to 
find that
\[
(\sup_I \delta_\Gamma^2) - (\inf_I\delta_\Gamma^2) \lesssim \int_{s\in I}  [1 - X_\gamma(\lambda(s)) \cdot \lambda'(s)] ds\ 
\le
 E_\gamma(\Lambda) \lesssim |\log\ep|^{-1}.
\]
It easily follows from this and \eqref{L2bound} that 
in fact $\delta_\Gamma(s) < R_\gamma$ for all $s$, when $\ep$ is small enough, 
and hence 
\eqref{W11}  holds.

Moreover, 
we integrate over \eqref{L2low}, use the definition of $E_{\gamma}(\Lambda,t)$ and the estimate \eqref{W11}, and get
\begin{equation}
\label{W12}
\int_{\R/\Z} |\gamma'(\sigma(s)) - \lambda'(s)|^2\, ds \lesssim |\log\eps|^{-1}.
\end{equation}

The two statements of the theorem, estimates \eqref{est1} and \eqref{est2}, now follow from \eqref{W11} and \eqref{W12} via
\begin{equation}
\label{est3}
\sup_{s\in \R/\Z}|\sigma(s) - (s + \bar \sigma)| \lesssim |\log\eps|^{-1}
\end{equation}
with  $\bar \sigma = \bar \sigma(t) = \sigma(0,t)$, because $\|\gamma'\|_{L^{\infty}} + \|\gamma''\|_{L^{\infty}}<\infty$. It thus remains to prove \eqref{est3}.
For {\em a.e.} $s\in \R/\Z$ we find from  \eqref{nablazeta} that
\[
\sigma'(s) =\nabla\zeta(\lambda(s))\cdot \lambda'(s) = \frac{\gamma'(\sigma(s)) \cdot \lambda'(s)}{1- (\lambda(s) - \gamma(\sigma(s))\cdot \gamma''(\sigma(s))}.
\]
Recalling that $|\gamma''(s)| \le r_\gamma(s) \le R_\gamma$, we use \eqref{W11} to deduce 
\[
| 1 - \sigma'(s)| \le   1 - \gamma'(\sigma(s))\cdot\lambda'(s) + \mathcal{O}(|\log\ep|^{-1}) .
\]
Thus $\| 1 - \sigma'\|_{L^1}\lesssim E_\gamma(\Lambda)\lesssim |\log\ep|^{-1}$.
Using the continuous embedding of $W^{1,1}$ into $L^{\infty}$ we find \eqref{est3}. This completes the proof of Theorem \ref{T2}.
\end{proof}

\subsection{Proof of Proposition \ref{Prop.ve}}\label{S:prototype}

In this subsection, we provide the proof of Proposition \ref{Prop.ve}. Notice that the statement in \eqref{ve.0} was established in Lemma \ref{L6} in Section \ref{S:Flatnorm}. The remaining estimates \eqref{ve.1} will be proved in Lemma \ref{L11} and estimates \eqref{ve.2} and \eqref{ve.4} will be proved in Lemma \ref{L8}.

In our computations we will occasionally  encounter error terms of the form $\| g(\kappa^*)\|_{L^1(\R/\Z)}$,
where for example $g(t) = |\log t|^p$ for some $p\ge 1$.
These can always be absorbed into the $\|\kappa^*\|_{L^{1,\infty}}$ term, since
(recalling that $\kappa^*\ge 1$ everywhere) we have
\begin{align}
\int g(\kappa^*(s))ds 
&= \int_1^\infty g'(\alpha) | \{ s \in \R/\Z : \kappa^*(s) \ge \alpha \}|\,  d\alpha
\nonumber\\
&\le \|\kappa^*\|_{L^{1,\infty}}\int_1^{\infty} \frac{ g'(\alpha) }\alpha\,  d\alpha
\ \ \lesssim \ \ \|\kappa^*\|_{L^{1,\infty}}
\label{g.est}\end{align}
by the virtue of the coarea formula \cite[Ch.\ 3.4]{EvansGariepy92}.

We now start to establish  pointwise estimates of $\ve$.  
We begin with rather crude estimates that are valid everywhere;
these will be sufficient for $L^q$ estimates of $\ve$, for $q>2$. 
For $q=2$,
we will later prove sharper estimates in the tube $\mathcal T$.

\begin{lemma}\label{L.pw0}
For every $x\in \R^3$,
\[
|\ve(x)| \lesssim \min \left\{ \frac 1{\dist (x,\Gamma)^2}, \ \frac{ \|\kappa^*\|_{L^{1,\infty}} }{ \dist (x,\Gamma)}, \ \frac 1 \ep \|\kappa^*\|_{L^{1,\infty}} \right\} .
\]
\end{lemma}

\begin{proof}
Notice first that the $\ve$ can be written as
\[
\ve(x) = \int_{\Gamma} \K^{\eps}(x-z)\times \tau_{\Gamma}(z)\, d\Ha^{-1},
\]
where $\K^{\eps} = \re\ast\K$, and $\K(z) =- \frac{z}{4\pi |z|^3}$ for $z\in \R^3\setminus\{0\}$ is the gradient of the Newtonian potential in $\R^3$.
The mean value property for harmonic functions implies that $\K^\ep(x) = \K(x)$ if $|x| > \ep$.
If $|x|\le \ep$, then
\[
|\left(\re\ast\K\right)(x)| \lesssim \int \frac{\re(x-y)}{|y|^2}\, dy \lesssim  \frac1{\eps^3} \int_{B_{2\eps}(0)}\frac1{|y|^2}\, dy\lesssim \frac1{\eps^2} \ .
\]
In particular,
\[
|\K^\ep(x)| \lesssim \min \left\{ |x|^{-2}, \ep^{-2}\right\}.
\]
It easily follows $|\ve(x)| \lesssim \dist(x,\Gamma)^{-2}$ for every $x$.

We fix $x$ and write $\delta := \dist (x,\Gamma)$. We first assume that $\delta \ge \eps$.
Then by Lemma \ref{L.linear}, because $|\K^{\eps}(x)|\lesssim |x|^{-2}$,
\begin{align*}
|\ve(x)| &\lesssim
\sum_{j=0}^\infty \int_{ \{ s : 2^j\delta \le |\gamma(s)- x| < 2^{j+1}\delta\} } |x- \gamma(s)|^{-2} \, ds
\\
&\lesssim
\| \kappa^*\|_{L^{1,\infty}}
\sum_{j=0}^\infty  (2^j\delta )^{-1} \lesssim 
\frac 1 \delta \| \kappa^*\|_{L^{1,\infty}}.
\end{align*}
Hence $|\ve (x)|\lesssim \delta^{-1}\|\kappa^*\|_{L^{1,\infty}}$ if $\delta\ge \eps$. 
Otherwise, if $\delta <\ep$ then we 
similarly appeal to Lemma \ref{L.linear} to find that
\begin{align*}
|\ve(x)| &\lesssim
\int_{ \{ s : |\gamma(s)- x| < \ep \} } \ep^{-2} \, ds + \sum_{j=0}^\infty \int_{ \{ s : 2^j\ep  \le |\gamma(s)- x| < 2^{j+1}\ep \} }
 |x-\gamma(s)|^{-2} \, ds
\\
&\lesssim
\frac 1 \ep\| \kappa^*\|_{L^{1,\infty}}.
\end{align*}
This proves the lemma.
\end{proof}

We can now establish $L^q$ estimates of $\ve$ for $q>2$.

\begin{lemma}\label{L11}
Estimates \eqref{ve.1} hold.
\end{lemma}

\begin{proof}
Inequality \eqref{ve.1} in the case $q=\infty$ is
already contained in the previous lemma.
For $q<\infty$, we will write $H^0(\Gamma) := \{x\in \R^3 : \dist(x,\Gamma)<  \ep\}$
and
\[
H^j(\Gamma) := \{x\in \R^3 : 2^{j-1}\ep \le \dist(x,\Gamma)<  2^j\ep\},\qquad \mbox{ for }j\ge 1.
\]
Then clearly
\[
\| \ve\|_{L^q}^q
=
\sum_{j=0}^\infty \int_{H^j(\Gamma)}|\ve|^q\,dx\ .
\]
We set  $N_r(\Gamma) :=  \{x\in \R^3 : \dist(x,\Gamma)< r \}$ and note that for every $r>0$,
\begin{equation}
| N_r(\Gamma) | \lesssim r^2 + r^3.
\label{vol.est}\end{equation}
Indeed, if we let $M = \lfloor 1/r \rfloor$, then
\[
N_r(\Gamma) \subset \bigcup_{k=0}^M B_{2r}(p_k)\qquad\mbox{where } p_k := \gamma(kr).
\]
Thus $|N_r(\Gamma)| \lesssim (M+1) r^3 \lesssim (1/r+1) r^3$, proving \eqref{vol.est}.
We now fix $J$ such that $2^J\ep \le 1 \le 2^{J+1}\ep$, and
we use Lemma \ref{L.pw0}
to estimate
\[
|\ve|\lesssim \min\left\{(2^j\eps)^{-2}, (2^j\eps)^{-1} \|\kappa^*\|_{L^{1,\infty}}\right\}\quad\mbox{in }H^j(\Gamma).
\]
Moreover, with the help of 
\eqref{vol.est} we obtain
\[
|H^j(\Gamma)|\le |N_{2^j\eps}(\Gamma)| \lesssim \begin{cases} (2^j\eps)^2 &\mbox{if }0\le j\le J, \\(2^j\eps)^3&\mbox{if }j>J.
\end{cases}
\]
Thus
\[
\int_{H^j(\Gamma)}|\ve|^q dx \ \lesssim 
\begin{cases}
\ep^{2-q} \|\kappa^*\|_{L^{1,\infty}}^q 2^{-j(q-2)} &\mbox{if }0\le j\le J,\\
\ep^{3-2q} 2^{-j(2q-3)} &\mbox{if }j > J.
\end{cases}
\]
We thus obtain \eqref{ve.1} by summing over $j$. 
\end{proof}

We require sharper estimates for the $L^2$ norms of $\ve$ and associated 
quantities, and for
these we establish a more precise description of $\ve$ in the tube $\mathcal T$.

\begin{lemma}
If $x\in \mathcal T$and dist$(x,\Gamma)\ge \ep$ then
\[
\left|\ve(x)-\frac{1}{2\pi} \frac{(\gamma(\zeta(x))-x)\times\gamma'(\zeta(x))}{\dist(x,\Gamma)^2}\right|\lesssim \frac 1{r(\zeta(x))}  |\log \dist(x,\Gamma)| + \frac {\|\kappa^*\|_{L^{1,\infty}}}{r(\zeta(x))}.
\]
\label{P.veptwise}\end{lemma}


\begin{proof}
Fix $x\in \mathcal T$ with $\dist (x,\Gamma) > \ep$.
We use the same notation as in the proof of Lemma
\ref{L.pw0}, and recall that $\K^\ep(x) = \K(x)$ for $|x|\ge \ep$. In particular,
\[
\ve(x) =
 \frac1{4\pi}\int_{-1/2}^{1/2} \frac{\gamma(s)-x}{|\gamma(s)-x|^3}\times\gamma'(s)\, ds.
\]
For notational convenience, we assume in the following  discussion that $\zeta(x)=0$ and we set $\delta: =\dist(x,\Gamma)$
and $r_0 := r(0)=r(\zeta(x))$. 
Then defining $\gamma_0(s) = \gamma(0)+ s\gamma'(0)$, we see
that 
\begin{eqnarray}
4\pi \ve(x) 
&= &
\int_{-r_0}^{r_0}  \frac{\gamma_0(s)-x}{|\gamma_0(s)-x|^3}\times\gamma_0'(s)\, ds +
\int_{-r_0}^{r_0} F(s) ds
\nonumber\\
&&\mbox{} + 
 \int_{  r_0<|s|<1/2 }  \frac{\gamma(s)-x}{|\gamma(s)-x|^3}\times\gamma'(s)\, ds,
\label{ve.split}
\end{eqnarray}
where
\[
F(s) :=f(\gamma(s),\gamma'(s)) - f(\gamma_0(s), \gamma_0'(s)),\qquad
\mbox{ for }f(p,\tau) := 
 \frac{p-x}{|p-x|^3}\times \tau.
\]
The last integral in \eqref{ve.split} can be estimated by exactly the arguments in the proof
of Lemma \ref{L.pw0}, leading to
\[
\left|  \int_{  r_0<|s|<1/2 }  \frac{\gamma(s)-x}{|\gamma(s)-x|^3}\times\gamma'(s)\, ds\right|
\lesssim
\frac 1 {r_0} \|\kappa^*\|_{L^{1,\infty}}.
\]
To evaluate the first integral in \eqref{ve.split},
we observe that since $\left(\gamma(0)-x\right)\cdot\gamma'(0)=0$, 
\[
\int_{-r_0}^{r_0}  \frac{\gamma_0(s)-x}{|\gamma_0(s)-x|^3}\times\gamma_0'(s)\, ds 
=  (\gamma(0)-x)\times \gamma'(0) 
\int_{-r_0}^{r_0} \frac 1{(\delta^2 + s^2)^{3/2}  } ds
\]
with
\[
\int_{-r_0}^{r_0} \frac 1{(\delta^2 + s^2)^{3/2}  } ds
= \frac 2{\delta^2} \left( 1+ \left(\frac \delta {r_0}\right)^2\right)^{-1/2}
= \frac 2{\delta^2} + \mathcal{O}\left(r_0^{-2}\right)
\]
for $\delta \le r_0$. So it only remains to estimate the second integral in \eqref{ve.split}.
By the mean value theorem (of calculus), we may write the integrand as
\[
F(s) =
\nabla_{p,\tau} f( \gamma_{\lambda}(s), \gamma_{\lambda}'(s)) \cdot(\gamma(s)-\gamma_0(s), \gamma'(s) - \gamma_0'(s))
\]
where $\gamma_{\lambda}(s) = \lambda\gamma(s) + (1-\lambda)\gamma_0(s)$ for some $\lambda$ between $0$ and $1$. 
Straightforward calculations then imply that
\begin{align*}
|F(s)|
&\lesssim \frac {|\gamma(s) - \gamma_0(s)|}{| \gamma_{\lambda}(s) - x|^3} +
\frac{|\gamma'(s) -\gamma_0'(s)|}{|\gamma_{\lambda}(s) - x|^2}.
 \end{align*}
Since $\gamma_0'(s) = \gamma'(0)$ for all $s$, it follows from the definition of $r_0$ that  
\[
|\gamma'(s) - \gamma_0'(s)|   \le \frac {|s|}{r_0}
\]
for all $|s|\le r_0 $, and hence that
\begin{equation}
|\gamma(s) - \gamma_0(s)| \le \frac {s^2}{2r_0}, \qquad
| \gamma_{\lambda}(s) - x| \ge \frac 12 (\delta^2+ s^2)^{1/2}. 
\label{need.later}
\end{equation}
We may therefore complete the proof by estimating the integral as
\[
\left|\int_{-r_0}^{r_0}F(s)\,ds \right| \lesssim \frac 1 {r_0}\left(\log \left(\frac {r_0}{\delta}\right)+1\right),
\]
because then
\[
\left|\ve(x) - \frac1{2\pi} \frac{(\gamma(\zeta(x))-x)\times\gamma'(\zeta(x))}{\dist(x,\Gamma)^2}\right|\lesssim \frac1{r_0}\left(\log\left(\frac{r_0}{\delta} \right) +1\right) + \frac{\delta}{r_0^2} + \frac{\|\kappa^*\|_{L^{1,\infty}}}{r_0}.
\]
Since $\delta \le r_0$ and $\|\kappa^*\|_{L^{1,\infty}}\ge 1$, the statement follows.
\end{proof}

\begin{lemma}\label{L3}
For every $\ep \in (0,\frac 12)$
\[
\|\ve\|_{L^2}^2  \le  \frac{1}{2\pi} |\log \eps| + \mathcal{O}(\| \kappa^*\|_{L^{1,\infty}}^2) \, .
\]
\end{lemma}

\begin{proof}
Let $\G$ denote  the Newtonian potential $\G(z) = \frac1{4\pi |z|}$,
and define $\Phi := \G\ast \mu_\Gamma$, so that $-\laplace \Phi = \mu_{\Gamma}$. Then we can write $\ve = \re\ast \nabla\times \Phi = \nabla\times \re \ast \Phi$.
It is easy to see, by arguing as in the proof of Lemma \ref{L.pw0}, that $|\Phi(x)| \lesssim |x|^{-1}$ for $|x|$ large, and together with 
the conclusions of Lemma \ref{L.pw0}, this gives sufficient decay to justify integrating by parts as follows:
\begin{multline*}
\int|\ve|^2 dx= \int  \curl( \re \ast\Phi)\cdot \ve\, dx  = \int  \re \ast \Phi \cdot \curl \ve \, dx \\ = \int \re \ast \Phi \cdot \re \ast \mu_\Gamma
\, dx
=
\int \re\ast\re\ast \Phi \cdot d\mu_\Gamma .
\end{multline*}
In the last identity, we have used the radial symmetry of $\re$.
Setting $\eta^\ep := \re \ast\re$, it follows that 
\begin{equation}\label{veL2.1}
\int |\ve|^2\, dx  \le 
\int_{\R/\Z} \left|\eta^\ep \ast \Phi (\gamma(s)) \right | \, ds
=\int_{\R/\Z} \left|(\eta^\ep \ast  \G \ast \mu_\Gamma) (\gamma(s)) \right|  \, ds.
\end{equation}
Below we will repeatedly use the facts that 
\begin{equation}
\eta^{\eps}\ast\G(z) = \G(z) \quad \mbox{ for every $|x|> 2\eps$},\qquad
\qquad
\eta^\eps\ast\G \lesssim \frac  1 \ep \ \  \mbox{  everywhere.}
\label{etaastG}\end{equation}
The first of these follows from the mean value property for harmonic functions, and the second
is easy to verify.

Now consider an arbitrary point in $\R/\Z$, which we take for convenience to be $s=0$,
and let $x:= \gamma(0)$.
Then
\[
\left|(\eta^{\eps}\ast\G\ast\mu_{\Gamma})(x)\right| \le \int_{-1/2}^{1/2} (\eta^{\eps}\ast\G)(x-\gamma(s))\, ds.
\]
Let $r_0 := r(0)$. If $r_0<4\ep$, then we use \eqref{etaastG} and  Lemma \ref{L.linear} to compute
\begin{align}
\left|(\eta^{\eps}\ast\G\ast\mu_{\Gamma})(x)\right| 
&\lesssim
\int_{ \{ s : |\gamma(s) - x|\le 2\ep \} } \frac 1 \ep \, ds \label{r0small}
+ \sum_{j= 1}^J
\int_{ \{ s :  2^j\ep \le |\gamma(s) - x| \le 2^{j+1}\ep  \} } \frac 1{|\gamma(s) - x|}\,ds
\nonumber  \\ &
\lesssim  \| \kappa^*\|_{L^{1,\infty}} |\log\ep|,
\end{align}
where $J\lesssim |\log\eps|$ because $|\gamma(s) - x|\le1$ for all $s$. 
For $r_0\ge 4\ep$ we proceed very much as in the proof of Lemma \ref{P.veptwise}, writing
\begin{multline*}
 \int_{-1/2}^{1/2} (\eta^{\eps}\ast\G)(x-\gamma(s))\, ds
=
 \int_{-r_0}^{r_0} (\eta^{\eps}\ast\G)(s \gamma'(0))\, ds \ +\  \int_{-r_0}^{r_0} F(s)\, ds
 \\
 +
 \int_{r_0 <|s|\le \frac 12}  (\eta^{\eps}\ast\G)(x-\gamma(s))  \, ds
\end{multline*}
where
\[
F(s) =  (\eta^{\eps}\ast\G)(x-\gamma(s)) -  (\eta^{\eps}\ast\G)(x - \gamma_0(s)),\qquad \gamma_0(s) = x+s\gamma'(0).
\]
Arguing as in the proof of \eqref{r0small} above, it follows from  \eqref{etaastG} and Lemma  \ref{L.linear}
that 
\[ 
 \int_{r_0 <|s|\le \frac 12}  (\eta^{\eps}\ast\G)(x- \gamma(s))  \, ds 
\ \lesssim  \  |\log r_0| \|\kappa^*\|_{L^{1,\infty}}.
\]
Next, again appealing to \eqref{etaastG}, it is straightforward to  check that
\[
\int_{-r_0}^{r_0} (\eta^{\eps}\ast\G)(s \gamma'(0))\, ds \  = 
 \  \frac {|\log (r_0/\ep)|}{2\pi} +\mathcal{O}(1)
 =  
\frac {|\log \ep|}{2\pi} + \mathcal{O}( |\log r_0|)
\]
where we have used the fact that $r(s)\le \frac 12$ for all $s$ to simplify the error terms.

Notice that in view of the second estimate in \eqref{need.later} (applied both with $\lambda=0$ and $\lambda=1$), we have $F(s) = \G(x-\gamma(s)) - \G(x-\gamma_0(s))$  for $4\ep \le |s| \le r_0$. Hence, there
exists some $ \gamma_{\lambda}(s)$, a
convex combination of $\gamma_0(s)$ and $\gamma(s)$, such that
\[
F(s) \ = \  \nabla \G(x-\gamma_{\lambda}(s))\cdot (\gamma(s) -\gamma_0(s))\ \lesssim \ 
\frac{|\gamma(s)-\gamma_0(s)|}{|x- \gamma_{\lambda}(s)|^2}.
\]
%
Again using \eqref{need.later}, we find that $F(s) \lesssim   \frac 1{r_0} $ if $4\ep \le |s| \le r_0$.
Since $F(s) \lesssim \frac 1 \ep$ trivially by \eqref{etaastG} for all $s$, we thus obtain
\[
\int_{-r_0}^{r_0} F(s)\, ds  \lesssim 1.
\]
Combining these, we find that if $r_0\ge 4\ep$, then
\[
|(\eta^\ep \ast \G \ast \mu_\Gamma)(\gamma(s))| \le \frac{|\log\ep|}{2\pi} + \mathcal{O}\left(|\log r(s)|\|\kappa^*\|_{L^{1,\infty}}\right).
\]
Recalling \eqref{veL2.1} and \eqref{r0small}, we can now integrate and recall \eqref{g.est}  to find that
\begin{eqnarray*}
\lefteqn{\int|\ve|^2 \, dx}\\
 &\le& \frac{|\log \eps|}{2\pi} + C\int |\log \kappa^*(s)|\, ds \|\kappa^*\|_{L^{1,\infty}} \\
 & & \hspace{5em} + C
 \|\kappa^*\|_{L^{1,\infty}}|\log \eps| \ \left|\left\{ s\in \R/\Z: \kappa^*(s)>\frac1{4\eps}\right\}\right|\\
 &\le& \frac{|\log\eps|}{2\pi} +C \|\kappa^*\|_{L^{1,\infty}}^2 + C\eps |\log\eps|\|\kappa^*\|^2_{L^{1,\infty}}\ .
\end{eqnarray*}
The statement follows because $\eps|\log\eps| \le 1$.
\end{proof}

The following Lemma completes the proof of Proposition \ref{Prop.ve}.

\begin{lemma}\label{L8}
Estimates \eqref{ve.2} and \eqref{ve.4} hold.
\end{lemma}

\begin{proof}
We first claim that it suffices to show that
\begin{equation}
\frac{4\pi}{|\log \eps|}\int_{\mathcal T} \phi: v^{\eps}\otimes v^{\eps}\, dx  =  \int_{\Gamma} \phi:(I-\tau\otimes\tau)\, d\Ha^1 
+ \mathcal{O}\left(\frac{\|\kappa^*\|_{L^{1,\infty}}^2 \|\phi\|_{W^{1,\infty}}}{|\log\eps|}\right).
\label{Tube.est}\end{equation}
Indeed, if this holds, then we may take $\phi  = I$ in \eqref{Tube.est} to find that
\[
\int_{\mathcal T} |\ve |^2 \, dx = \frac{ |\log\ep|}{2\pi}
+ \mathcal{O}\left(\|\kappa^*\|_{L^{1,\infty}}^2 \right).
\]
This, together with Lemma \ref{L3}, implies that
\[
\int_{\R^3\setminus \mathcal T} |\ve|^2\, dx  =  \mathcal{O}\left(\|\kappa^*\|_{L^{1,\infty}}^2 \right),
\]
and from this we see that \eqref{Tube.est} implies \eqref{ve.4}.
Similarly, combining the previous two estimates, we directly obtain \eqref{ve.2}.

To prove \eqref{Tube.est}, we first use the coarea formula to rewrite the integral on the left-hand
side as
\begin{equation}\label{coarea.tube}
\int_{\mathcal T} \phi: v^{\eps}\otimes v^{\eps}\, dx  =
\int_{\R/\Z} \left(\int_{\zeta^{-1}(s) } \phi: v^{\eps}\otimes v^{\eps}\ |\nabla \zeta|^{-1} d\calH^2 \right) ds.
\end{equation}
We now consider some $s\in \R/\Z$.
It is convenient to choose coordinates so that $\gamma(s) = 0$ and
$\gamma'(s) =(0,0,1)$.  We will also write $r = (x_1^2 + x_2^2 )^{1/2}$,
and we remark that $r = \dist(x,\Gamma)$ in $\zeta^{-1}(s)$.
In these coordinates,
\[
\zeta^{-1}(s)
=
 \left\{  x : x_3 = 0,  r < \frac14r(s) \right\},  
 \]
and for $x$ in this set, according to Lemma \ref{L.pw0}  
\[ 
|\ve(x)|\lesssim  \| \kappa^*\|_{L^{1,\infty}} \min\left\{ \frac 1\ep, \frac 1 r\right\}.\\
\]
Moreover, it follows from  Lemma \ref{L9cis} and Fubini's Theorem
that for {\em a.e.} $s$, 
\[
\big| \,|\nabla \zeta(x)|^{-1}-1 \big| \le   \frac r { r(s)}  = r \kappa^*(s) \quad\qquad\mbox{for  {\em a.e. } }x\in \zeta^{-1}(s) \ .
\] 
We henceforth restrict our attention to $s$ for which this holds.
We therefore have $|\ve|\lesssim  \eps^{-1} \|\kappa^*\|_{L^{1,\infty}}$ if $r\le \ep$, and otherwise
\[
| \phi(x): v^{\eps}\otimes v^{\eps}\ |\nabla \zeta|^{-1} - \phi(0):v^{\eps}\otimes v^{\eps}| \ \lesssim  \frac 1 r
\|\kappa^*\|^2_{L^{1,\infty}} \left(\frac {\|\phi\|_{L^{\infty}} }{r(s)} + \|\nabla \phi\|_{L^{\infty}}\right).
\]
It follows that 
\begin{multline} \label{537}
\int_{\zeta^{-1}(s) } \phi: v^{\eps}\otimes v^{\eps}\ |\nabla \zeta|^{-1} d\calH^2
=
\int_{ \zeta^{-1}(s)\setminus B_\ep }
\phi(0): \ve \otimes \ve \ d\calH^2  \\+ 
\mathcal{O}\left(\|\kappa^*\|^2_{L^{1,\infty}} (\|\phi\|_{L^{\infty}} + r(s) \|\nabla \phi\|_{L^{\infty}}) \right).
\end{multline}
Next, the estimates in Lemma \ref{P.veptwise} imply that for $v_*(x) := \frac 1 {2\pi}\frac{(-x_2,x_1, 0)}{r^2} $, 
we have
\[
\left|
\ve\otimes \ve - v_*\otimes v_*
\right|
\lesssim \left(\frac{ |\log r|}{r(s)} + \frac{\|\kappa^*\|_{L^{1,\infty}}}{r(s)}\right)\frac{\|\kappa^*\|_{L^{1,\infty}}}{r}
\]
for $\ep < r < r(s)$. So  we integrate to find that
\begin{eqnarray}
\lefteqn{\int_{ \zeta^{-1}(s)\setminus B_\ep }
 \ve \otimes \ve \ d\calH^2 }\nonumber \\
 &=&\int_{ \zeta^{-1}(s)\setminus B_\ep }
v_*\otimes v_* \ d\calH^2 
+ \mathcal{O}\left(\|\kappa^*\|_{L^{1,\infty}}|\log\kappa^*(s)|+ \|\kappa^*\|_{L^{1,\infty}}^2\right).\label{vep.vstar}
\end{eqnarray}
For example, one of the two error terms is estimated by
\begin{eqnarray*}
\frac{\|\kappa^*\|_{L^{1,\infty}}}{r(s)}\int_{ \zeta^{-1}(s)\setminus B_\ep } \frac {|\log r|}r d\calH^2 
& \sim & \frac{\|\kappa^*\|_{L^{1,\infty}}}{r(s)} \int_\ep^{r(s)} |\log r|\, dr\\
& \lesssim & \|\kappa^*\|_{L^{1,\infty}}|\log(r(s))|  = \|\kappa^*\|_{L^{1,\infty}}|\log \kappa^*(s)|. 
\end{eqnarray*}
The other terms is similar. Moving on, it is easy to check that
\begin{equation}\label{vstar.est}
\int_{ \zeta^{-1}(s)\setminus B_\ep }
v_*\otimes v_* \ d\calH^2 
=
 \frac 1{4\pi} \log\left( \frac{r(s)}{\ep}\right) \left(
\begin{array}{ccc}1&0&0\\0&1&0\\0&0&0
\end{array}
\right).
\end{equation}
Indeed, it is clear that any term involving the $3$rd component of $v_*$ must vanish.
Among the remaining terms,  symmetry considerations imply that the off-diagonal
terms 
vanish and that the diagonal terms are equal. Since their sum is
\[
\int_{\zeta^{-1}(s)\setminus B_\ep} |v_*|^2d\calH^2 = 
\frac  1{2\pi} \int_\ep^{r(s)} \frac 1r  \ dr =  \frac1 {2\pi} \log\left(\frac {r(s)}{\ep}\right),
\]
the claim \eqref{vstar.est} follows.

Now by combining \eqref{537}, \eqref{vep.vstar} and \eqref{vstar.est}, 
and recalling that $r(s) = \frac 1{\kappa^*(s)} \le \frac 12$ for all $s$,
we find that
\begin{multline*}
\left| \int_{\zeta^{-1}(s) } \phi: v^{\eps}\otimes v^{\eps}\ |\nabla \zeta|^{-1} d\calH^2
 \ - \ \frac {|\log \ep|} {4\pi}  \phi(\gamma(s)):(I - \gamma'(s)\otimes \gamma'(s))\right|
\\
\lesssim
\left(|\log \kappa^*(s)|\|\kappa^*\|_{L^{1,\infty}}+ \|\kappa^*\|^2_{L^{1,\infty}} \right) \| \phi\|_{W^{1,\infty}}. 
\end{multline*}
We deduce \eqref{Tube.est}, and hence complete the proof of the lemma, by substituting this
into \eqref{coarea.tube}, integrating and using \eqref{g.est} to simplify some of the error terms.
\end{proof}

\section*{Appendix}

In this appendix, we prove the proof of the interpolation inequality
\[
\|f\|_{L^p} \lesssim \|f\|_{L^{1,\infty}}^{\frac{2-p}p} \|f\|_{L^2}^{\frac{2p-2}p}.
\]
Recall that $\|f\|_{L^p}^p  = p\int_0^{\infty} \alpha^{p-1} d_f(\alpha)\, d\alpha$, where $d_f(\alpha) = \left|\left\{x\in\R^N:\: |f(x)|>\alpha\right\}\right|$. Then, letting $\ell>0 $ be arbitrary, we write
\[
\int |f|^p\, dx = p \int_0^{\ell} \alpha^{p-1} d_f(\alpha)\, d\alpha + p \int_{\ell}^\infty \alpha^{p-1}d_f(\alpha)\, d\alpha.
\]
Clearly
\[
\int_0^{\ell} \alpha^{p-1} d_f(\alpha)\,d\alpha \le \ell^{p-1} \|f\|_{L^{1,\infty}},\quad\mbox{and}\quad \int_{\ell}^{\infty} \alpha^{p-1} d_f(\alpha)\, d\alpha \le \frac{\ell^{p-2}}2 \|f\|_{L^2}^2,
\]
where we have used that $1<p<2$. Hence,
\[
\|f\|_{L^p}^p \lesssim \ell^{p-1} \|f\|_{L^{1,\infty}} + \ell^{p-2} \|f\|_{L^2}^2.
\]
Optimizing in $\ell$ yields the desired result.

\section*{Acknowledgment}
The work of both authors  was partially supported by the National Science and
Engineering Research Council of Canada under operating grant 261955.

\bibliography{euler}
\bibliographystyle{acm}

\end{document}